%% file: hodge-local-high-arxiv.tex
\documentclass[10pt]{amsart}

\usepackage{amsmath, amsthm, amssymb, amsfonts, verbatim}

\usepackage[active]{srcltx}
\usepackage{array}
\usepackage[bookmarks]{hyperref}
\usepackage{amscd}
\usepackage{epsfig}
\usepackage{color}
\usepackage{comment}
\usepackage{setspace}
\usepackage{textcomp}
\usepackage{url}
\usepackage{multirow}
\usepackage{stmaryrd}
\usepackage{pdfsync}
\usepackage{breqn}

\numberwithin{equation}{section}

\input{./jlmacros.tex}

\newcommand{\mfh}{\mathfrak{H}}
\newcommand{\Lm}{\Lambda}

\newcommand{\Alt}{\operatorname{Alt}}
\newcommand{\vol}{\operatorname{vol}}
\newcommand{\dd} {{\rm d}}

\newcommand{\cdeg}{\operatorname{cdeg}}
\newcommand{\ncdeg}{\operatorname{ncdeg}}

\newcommand{\s}{\sigma}

\newcommand{\Q}{\mc{Q}}

\begin{document}

\title[Local coderivatives]{High order approximation of Hodge Laplace problems with local coderivatives on cubical meshes}

\author{Jeonghun J. Lee} 
\address{Department of Mathematics, Baylor University, Waco, TX , USA}
\email{jeonghun\_lee@baylor.edu}
\urladdr{}
\date{July, 2019 }
\maketitle

\begin{abstract}
In mixed finite element approximations of Hodge Laplace problems associated with the de Rham complex, the exterior derivative operators are computed exactly, so the spatial locality is preserved. 
However, the numerical approximations of the associated coderivatives are nonlocal and
it can be regarded as an undesired effect of standard mixed methods. 
For numerical methods with local coderivatives a perturbation of low order mixed methods in the sense of variational crimes has been developed for simplicial and cubical meshes.
In this paper we extend the low order method to all high orders on cubical meshes using a new family of finite element differential forms on cubical meshes.
The key theoretical contribution is a generalization of the linear degree, in the construction of the serendipity family of differential forms, and the generalization is essential in the unisolvency proof of the new family of finite element differential forms. 
\end{abstract}

\section{Introduction}
In this paper we consider finite element methods for the Hodge Laplace problems of the de Rham complex where both the approximation 
of the exterior derivative and the associated coderivative are spatially local operators. 
The locality of the coderivative operator is not fulfilled in standard mixed methods for these problems (cf. \cite{{AFW06, AFW10}}). 
In fact, pursuing numerical methods with local coderivative is related to the development of various numerical methods for the Darcy flow problems.

To discuss this local coderivative property in a more familiar context,
let us consider the mixed form of Darcy flow problems: Find $(\sigma, u) \in H(\div,\Omega) \times L^2(\Omega)$ such that 
\algn{
\label{mixed-darcy}
\begin{split}
\LRa{\bs{K}^{-1}\sigma, \tau} - \LRa{u, \div \tau} &= 0, \qquad \quad \forall \tau \in H(\div , \Omega),\\
\LRa{\div \sigma,v}  &= \LRa{f,v}, \quad \forall v \in L^2(\Omega),
\end{split}
}
where the unknown functions $\sigma$ and $u$ are vector and scalar fields defined on a bounded domain  $\Omega$ in  $\R^n$, and $\partial \Omega$ is its boundary.
The coefficient $\bs{K}$ is symmetric, matrix-valued, spatially varying, and uniformly positive definite. 
Note that $u$ is the scalar field of the pressure and $\sigma$ is the fluid velocity given by the Darcy law $\sigma = - \bs{K} \grad u$.  
The standard mixed finite element method for this problem is:

Find $(\sigma_h,u_h) \in \Sigma_h \times V_h$ such that 
\algn{
\label{mixed-Lap}
\begin{split}
\LRa{\bs{K}^{-1}\sigma_h, \tau} - \LRa{u_h, \div \tau} &= 0, \qquad \quad \forall \tau \in \Sigma_h,\\
\LRa{\div \sigma_h,v}  &= \LRa{f,v}, \quad \forall v \in V_h,
\end{split}
}
where $\Sigma_h \subset H(\div, \Omega)$ and $V_h \subset L^2(\Omega)$ are finite element spaces, and $\sigma_h$ is an approximation of
$\sigma = - \bs{K}\grad u$. Here the notation $\LRa{\cdot, \cdot}$ is used to denote the $L^2$ inner product for both scalar fields 
and vector fields defined on $\Omega$.

The mixed method \eqref{mixed-Lap} has a local mass conservation property, 
and its stability conditions and error estimates are well-studied (cf. \cite{BFBook}). 
However, the mixed method \eqref{mixed-Lap} does not preserve the local property of the map $u \mapsto \sigma = - \bs{K} \grad u$ in the continuous problem. In other words, the map $u_h \mapsto \sigma_h$, defined by the first equation of \eqref{mixed-Lap}, is not local 
because 
the inverse of the so--called ``mass matrix" derived from the $L^2$ inner product 
$\LRa{\bs{K}^{-1}\tau, \tau'}$ on $\Sigma_h$ is nonlocal.
Since constitutive laws are spatially local relations of quantities in many physical models,
construction of local numerical constitutive laws is one of key issues in the development of numerical methods following physical derivation of constitutive laws such as the finite volume methods and the multi-point flux approximations.



Here we give a brief overview on previous studies of numerical methods with local coderivatives mainly for the Darcy flow problems but we have to admit that this overview and the list of literature here are by no means complete. 

An early work for the locality property by perturbing the mixed method \eqref{mixed-Lap} 
was done in \cite{Baranger-Maitre-Oudin-1996} on triangular meshes with 
the lowest order Raviart-Thomas space. 
The approach leads to a two-point flux method, which approximate flux across 
the interface of two cells by two point values of pressure field in the two cells, 
but 
the two-point flux method is not consistent in general for anisotropic $\bs{K}$,  cf. \cite{Aavatsmark-2002,Aavatsmark2007}.
To circumvent this defect, various  multi-point flux approximation schemes were derived (cf. \cite{Aavatsmark-2002}) but the stability and error estimates of these schemes are usually  nontrivial and are restricted to low order cases. 
It seems that the most useful approach for the stability and error estimates for these numerical schemes is 
to utilize connections between the schemes and perturbed mixed finite element methods, cf. \cite{Bause-Hoffman-Knabner-2010,Droniou-Eymard-2006,Klausen-Winther-2006b,Klausen-Winther-2006a,Vohralik-Wohlmuth-2013,Wheeler-Yotov-2006}.
An alternative approach to perturbed mixed finite element methods for the local coderivative property was proposed in \cite{Brezzi-Fortin-Maridi-2006,Wheeler-Yotov-2006} independently
for simplicial and quadrilateral meshes. The key to achieve local coderivatives in this approach is a mass-lumping for vector-valued finite elements. 
Further extensions to hexahedral grids are studied in  \cite{Ingram-Wheeler-Yotov-2010,Wheeler-Xue-Yotov-2012}. Extensions to all high order methods are studied in \cite{Lee-Yotov-etal} with the development of a new family of $H(\div)$ finite elements on quadrilateral and hexahedral meshes, which is inspired by the new family of low order finite elements in \cite{Lee-Winther}. For the Maxwell equations, Cohen and Monk studied perturbed mixed methods based on anisotropic mass-lumping of vector-valued finite elements but the methods may not be consistent when the material coefficients are not isotropic (cf.  \cite{MonkCohen1998}). 
For the Hodge Laplace problems the discrete exterior calculus, proposed in \cite{DEC05,Hirani-thesis}, cf. also \cite{Hiptmair-2001}, has a natural local coderivative property by construction. 
However, a satisfactory convergence theory seems to be limited (see \cite{Tsogtgerel-Schulz-2018} for 0-form case).

The purpose of this paper is to extend the results in 
 \cite{Lee-Yotov-etal,Lee-Winther} to the Hodge Laplace problems on cubical meshes. 
More precisely, we will construct high order perturbed mixed methods for the Hodge Laplace problems on cubical meshes which have the local numerical coderivatives. 
Since the Hodge Laplace problems are models of other specific problems, the numerical method in the present paper can be used to develop numerical methods with the locality property for other problems. For example, the methods with the newly developed finite elements have potential applications to high order mass-lumping for the time-dependent Maxwell equations.
However, we will restrict our discussion only to steady problems in the paper because developing the methods for steady problems is already quite involved. 

The paper is organized as follows. In the next section we will present a brief review
of exterior calculus, the de Rham complex with its discretizations, and the abstract analysis results in \cite{Lee-Winther} for the analysis framework to be used in later sections.  In Section 3 we develop a new family of finite element differential forms on cubical meshes, say $\tilde{\mc{Q}}_r \Lm^k$, by defining the shape functions and the degrees of freedom, and proving the unisolvency. We also prove some properties of the new elements for construction of numerical methods with the local coderivative property. 
In Section~4, we construct numerical methods with local coderivatives using $\tilde{\mc{Q}}_r \Lm^k$ and the framework in Section~2. 
Finally, we summarize our results with some concluding remarks in Section 5.

\section{Preliminaries}\label{prelim}
Here we review the language of the finite element exterior calculus \cite{{AFW06, AFW10}} and also introduce new concepts of polynomial differential forms.  We assume that $\Omega \subset \R^n$ is a bounded
domain with a polyhedral boundary. 
We will consider finite element approximations of a differential equation which has a differential form defined on $\Omega$ as an unknown.
Let $\Alt^k(\R^n)$ be the space of alternating $k$--linear maps on $\R^n$. 
For $1 \leq k \leq n$ let $\Sigma_k$ be the set of increasing injective maps from $\{1, ..., k\}$ to $\{1, ..., n \}$. Then 
we can define an inner product on $\Alt^k(\R^n)$ by 
\[
\LRa{ a,b}_{\Alt} = \sum_{\sigma \in \Sigma_k}
a(e_{\sigma_1},\ldots,e_{\sigma_k})b(e_{\sigma_1},\ldots,e_{\sigma_k}),
\quad a,b \in \Alt^k(\R^n),
\]
where $\sigma_i$ denotes $\sigma(i)$ for $1 \leq i \leq k$ and $\{e_1,\ldots,e_n\}$ is any orthonormal basis of $\R^n$. 
The differential $k$-forms on $\Omega$ are maps defined on $\Omega$ with values in 
$\Alt^k (\R^n)$.
If $u$ is a differential $k$-form and $t_1,  \ldots ,t_k$ are vectors in $\R^n$, then $u_x(t_1, \ldots, t_k)$ denotes the value of 
$u$ applied to the vectors $t_1,  \ldots ,t_k$ at the point $x \in \Omega$. 
The differential form $u$ is an element of the  space $L^2 \Lm^k(\Omega)$ if and only if 
the map
\[
x \mapsto u_x(t_1, \ldots, t_k)
\]
is in $L^2(\Omega)$ for all 
tuples $t_1,  \ldots ,t_k$.
In fact, $L^2 \Lm^k(\Omega)$ is a Hilbert space with inner product given by
\[
\LRa{u,v} = \int_{\Omega} \LRa{u_x,v_x}_{\Alt} \, dx.
\]
The exterior derivative of a $k$-form $u$ is a $(k+1)$-form $\dd u$ given by
\[
\dd u_x(t_1,   \ldots t_{k+1}) = \sum_{j=1}^{k+1} (-1)^{j+1} \partial_{t_j} u_x(t_1, 
\ldots, \hat t_j, \ldots ,t_{k+1}),
 \]
 where $\hat t_j$ implies that $t_j$ is not included, and $\partial_{t_j}$ denotes the directional derivative.
The Hilbert space $H\Lambda^k(\Omega)$ is the corresponding  space of 
$k$-forms $u$ on $\Omega$, which is in $L^2 \Lm^k(\Omega)$, and where
its exterior derivative, $\dd u = \dd^ku$, is also in $L^2 \Lm^{k+1}(\Omega)$.  The $L^2$ version of the de Rham complex then takes
the form
 \algn{ \label{eq:de-Rham}
H\Lambda^0(\Omega)
\xrightarrow{\dd^0} H\Lambda^1(\Omega) \xrightarrow{\dd^1}
\cdots \xrightarrow{\dd^{n-1}}
H\Lambda^n(\Omega).
}
In the setting of $k$--forms, the  Hodge Laplace problem takes the form 
\begin{equation}\label{hodge-Lap-strong}
Lu = (\dd^* \dd + \dd \dd^*)u = f,
\end{equation}
where $\dd =\dd^k$ is the exterior derivative mapping $k$--forms to $(k+1)$--forms, and the coderivative  $\dd^*= \dd_k^*$ can be seen as the formal
adjoint of $\dd^{k-1}$. Hence, the Hodge Laplace operator $L$ above is more precisely expressed as $L = \dd_{k+1}^* \dd^k + \dd^{k-1} \dd_k^*$.
A typical model problem studied in \cite{AFW06, AFW10} is of the form \eqref{hodge-Lap-strong} and with appropriate boundary conditions.
The mixed finite element methods are derived from a weak formulation, where $\sigma = \dd^* u$ is introduced as an additional variable. It is of the form:

Find $(\sigma, u) \in H\Lambda^{k-1}(\Omega) \times H\Lambda^{k}(\Omega)$ such that 
\begin{align}\label{hodge-mixed}
\begin{split}
\LRa{\sigma, \tau} - \LRa{u, \dd^{k-1}\tau} &= 0, \qquad \quad \tau \in H\Lambda^{k-1}(\Omega), \\
\LRa{\dd^{k-1}\sigma,v} + \LRa{\dd^k u,\dd^k v} &= \LRa{f,v}, \quad v \in H\Lambda^{k}(\Omega).
\end{split}
\end{align}
Here  $\LRa{\cdot,\cdot}$ denotes the inner products of all the spaces of the form  $L^2 \Lm^j(\Omega)$ which appears in the formulation,
i.e., $j = k-1,k,k+1$.
We refer to Sections 2 and 7 of \cite{AFW06} for more details.
We note that  only the exterior derivate $\dd$ is used explicitly in the weak  formulation above, while the relation $\sigma = \dd_k^*u$ 
is formulated weakly in the first equation. 
The formulation also contains the proper natural boundary conditions.
The problem \eqref{hodge-mixed} with $k=n-1$ corresponds to a weak formulation 
of the elliptic equation \eqref {hodge-Lap-strong} in the case when the coefficient $\bs{K}$ is the identity matrix.
The weak formulations \eqref{hodge-mixed} can be modified for variable by changing the $L^2$ inner products to the inner product with the variable coefficient, see \cite[Section 7.3]{AFW06}.
Throughout the discussion  below we will restrict  the discussion 
to the constant coefficient case but the extension of the discussion to problems with variable coefficients which are piecewise constants with respect to 
the mesh, is straightforward. We refer to \cite[Section~6]{Lee-Winther} for details.

If the domain $\Omega$ is not homologically trivial, then 
there may exist nontrivial harmonic forms, i.e., nontrivial elements of the space 
\[
\mfh^k(\Omega) = \{ v \in H\Lm^k (\Omega) \;:\; \dd v = 0 \text{ and }\LRa{v, \dd \tau} = 0 \text{ for all }\tau \in H\Lm^{k-1} (\Omega) \} ,
\]
and the solutions of the system \eqref{hodge-mixed} may not be unique. 
To obtain a system with a unique solution, an extra condition requiring orthogonality with respect to the harmonic forms:

Find $(\sigma, u, p) \in H \Lm^{k-1} (\Omega) \times H \Lm^k(\Omega) \times \mfh^k(\Omega)$ such that
\begin{align}
\notag \LRa{{\sigma}, \tau} - \LRa{\dd \tau, {u}}  &= 0 , & &   \tau \in H \Lm^{k-1} (\Omega), \\
\label{eq:hodge-mixed-cont}\LRa{\dd {\sigma}, v} + \LRa{\dd {u}, \dd v} + \LRa{{p}, v} &= \LRa{f, v}, & &   v \in H \Lm^k (\Omega), \\
\notag \LRa{{u}, q} &= 0, & &   q \in \mfh^k (\Omega) .
\end{align}

The basic construction in finite element exterior calculus is
 a corresponding subcomplex of \eqref{eq:de-Rham},
\[
V_h^0\xrightarrow{\dd} V_h^1 \xrightarrow{\dd}
\cdots \xrightarrow{\dd}
V_h^n,
\]
where the spaces $V_h^k$ are finite dimensional subspaces of
$H\Lambda^k(\Omega)$. In particular, the discrete spaces should have the property that 
$\dd (V_h^{k-1}) \subset V_h^k$.
The finite element methods studied in \cite{AFW06, AFW10} are based on the weak formulation \eqref{hodge-mixed}.
These methods are obtained by simply replacing the Sobolev spaces $H\Lambda^{k-1}(\Omega)$ and $H\Lambda^{k}(\Omega)$ 
by the finite element spaces $V_h^{k-1}$ and $V_h^k$.
More precisely, we are searching for a triple $(\tilde{\sigma}_h,\tilde{u}_h, \tilde{p}_h) \in V_h^{k-1} \times V_h^k \times \mfh_h^k$ such that 
\begin{align}
\notag \LRa{\tilde{\sigma}_h, \tau} - \LRa{\dd \tau, \tilde{u}_h}  &= 0 , & &   \tau \in V_h^{k-1}, \\
\label{eq:hodge-mixed-h}\LRa{\dd \tilde{\sigma}_h, v} + \LRa{\dd \tilde{u}_h, \dd v} + \LRa{\tilde{p}_h, v} &= \LRa{f, v}, & &   v \in V_h^k, \\
\notag \LRa{\tilde{u}_h, q} &= 0, & &   q \in \mfh_h^k,
\end{align}
where the space $\mfh_h^k$, approximating the harmonic forms, is given by 
\[
\mfh_h^k = \{ v \in V_h^k \;:\; \dd v = 0 \text{ and }\LRa{v, \dd \tau} = 0 \text{ for all }\tau \in V_h^{k-1} \} .
\]
Stability and error estimates for the numerical solution of \eqref{eq:hodge-mixed-h} are discussed in \cite[Theorem~3.9]{AFW10} with an error estimate of the form 
\algn{ \label{eq:tilde-approx}
\| (\sigma, u, p) - (\tilde{\sigma}_h, \tilde{u}_h, \tilde{p}_h) \|_{\mc{X}} \quad \lesssim \inf_{(\tau,v,q) \in \mc{X}_h^k} \| (\sigma, u, p) - (\tau, v, q) \|_{\mc{X}} + \mc{E}_h(u) 
}
with 
\algns{
\| (\sigma, u, p) \|_{\mc{X}} := \LRp{ \| \sigma \|^2 + \| \dd \sigma \|^2 + \| u \|^2 + \| \dd u \|^2 + \| p \|^2}^{\frac 12} 
}
and $\mc{E}_h(u)$ comes from the nonconformity of the space of discrete harmonic forms, $\mfh_h^k$, to the space of continuous harmonic forms $\mfh^k$.
$\mc{E}_h(u)$ vanishes if there is no nontrivial harmonic forms, and will usually be of higher order than the 
other terms on the right-hand side of \eqref{eq:tilde-approx}. We refer to \cite[Section 7.6]{AFW06} and \cite[Section 3.4]{AFW10} for more details.

In \cite{Lee-Winther}, abstract conditions are proposed to develop numerical methods satisfying such properties and the lowest order methods were developed in simplicial and cubical meshes. In particular, the cubical mesh case required construction of new finite element differential forms which were called $\mc{S}_1^+ \Lm^k$ spaces. In \cite{Lee-Yotov-etal}, the $\mc{S}_1^+ \Lm^k$ family is extended to all higher orders for the Darcy flow problems in the two and three dimensions on cubical meshes, and as a consequence, higher order numerical methods satisfying the local coderivative property in the two and three dimensional Darcy flow problems on weakly distorted quadrilateral and hexahedral meshes are constructed. The goal of this paper is developing high order numerical methods for the Hodge Laplace equations satisfying the local coderivative property on cubical meshes. For the stability and error analysis we use the abstract framework in \cite{Lee-Winther}. The main contributions of this paper are construction of a new family of finite element differential forms, a higher order extensions of $\mc{S}_1^+ \Lm^k$, and the new techniques for the development of the new elements. Although the new family is a higher order extension of $\mc{S}_1^+ \Lm^k$, we will call it $\tilde{\mc{Q}}_r \Lm^k$ family because it is closely related to the $\mc{Q}_r \Lm^k$ space rather than the $\mc{S}_r\Lm^k$ family.
One characteristic feature of $\tilde{\mc{Q}}_r\Lm^k$ family is that the number of degrees of freedom is same as the one of $\mc{Q}_r\Lm^k$. 

Here we summarize the abstract conditions for stability and error estimates established in \cite{Lee-Winther}.
\begin{itemize}
 \item[\bf (A)] There is a symmetric bounded coercive bilinear form $\LRa{\cdot, \cdot}_h$ on $V_h^{k-1} \times V_h^{k-1}$ such that the norm $\| \tau \|_h:= \LRa{\tau , \tau}_h^{1/2}$ is equivalent to $\| \tau \|$ for $\tau \in V_h^{k-1}$ with constants independent of $h$. 
\end{itemize}

\begin{itemize}
\item[{\bf (B)}] There exist discrete subspaces $W_h^{k-1} \subset L^2 \Lm^{k-1}(\Omega)$ and $\tilde{V}_h^{k-1} \subset V_h^{k-1}$ such that 
\algn{ \label{eq:B-cond}
\LRa{\tau, \tau_0} = \LRa{ \tau, \tau_0}_h, \qquad \tau \in \tilde{V}_h^{k-1}, \tau_0 \in W_h^{k-1},
}
and a linear map $\Pi_h : V_h^{k-1} \ra \tilde{V}_h^{k-1}$ such that $ \dd \Pi_h \tau = \dd \tau $, $\| \Pi_h \tau \| \lesssim \| \tau \|$, and 
\algn{ \label{eq:B-cond2}
  \LRa{ \Pi_h \tau, \tau_0 }_h =        \LRa{ \tau , \tau_0}_h, \quad  \tau_0 \in W_h^{k-1}. 
}
\end{itemize}

If these assumptions are satisfied, then we find a solution $(\sigma_h, u_h, p_h)$ of the problem  
\begin{align} 
\notag \LRa{\sigma_h, \tau}_h - \LRa{\dd \tau, u_h}  &= 0 , & &   \tau \in V_h^{k-1}, \\
\label{eq:Bh-eq} \LRa{\dd\sigma_h, v} + \LRa{\dd u_h, \dd v} + \LRa{p_h, v} &= \LRa{f, v}, & &   v \in V_h^k, \\
\notag \LRa{u_h, q} &= 0, & &   q \in \mfh_h^k .
\end{align}
The following result is proved in \cite{Lee-Winther}. 

\begin{theorem} \label{thm:main} Suppose that the assumptions {\rm \bf (A)} and {\bf(B)} hold. 
Then the solution of \eqref{eq:Bh-eq},  $(\sigma_h, u_h, p_h)$, satisfies
\algn{ \label{eq:err-estm}
\| (\sigma_h - \tilde{\sigma}_h, u_h - \tilde{u}_h, p_h - \tilde{p}_h) \|_{\mc{X}} \lesssim \| \sigma - P_{W_h} \sigma \| + \| \sigma - \tilde{\sigma}_h \|,
}
where $P_{W_h}$ is the $L^2$-orthogonal projection into $W_h^{k-1}$.
\end{theorem}

\section{Construction of $\tilde{\mc{Q}}_r \Lm^k$} \label{sec:rect-eg}
In this section we construct a new family of finite element differential forms $\tilde{\mc{Q}}_r \Lm^k (\mc{T}_h)$ on cubical meshes $\mc{T}_h$ of $\Omega$
where the elements in $\mathcal{T}_h$ are Cartesian product of intervals.
If $r=1$, then $\tilde{\mc{Q}}_r \Lm^k$ is the $\mc{S}_1^+ \Lm^k$ space in \cite{Lee-Winther}. 
For $n=2, 3$ and $k = n-1$, $\tilde{\mc{Q}}_r \Lm^k$ spaces are the $H(\div)$ finite element spaces which were discussed in \cite{Lee-Yotov-etal}. The idea of these new elements construction in \cite{Lee-Winther} and \cite{Lee-Yotov-etal} is enriching the shape function space of the $\mc{Q}_1^- \Lm^k$ and the Raviart-Thomas-Nedelec elements on cubical meshes with $\dd$-free and divergence-free shape functions
in order to associate a basis of the shape function space to the degrees of freedom given by nodal point evaluation. The most technical difficulty of new element construction in \cite{Lee-Yotov-etal} is the unisolvency proof. For this,  the authors in \cite{Lee-Yotov-etal} used some features of the enriched shape functions observed from their explicit expressions.   
However, it is highly nontrivial to use the same argument to $\tilde{\mc{Q}}_r \Lm^k$ for general $n$ and $k$ because explicit forms of the enriched shape functions for general $n$ and $k$ are too complicated to use conventional unisolvency arguments.  
To circumvent this difficulty we will introduce new quantities of polynomial differential forms which allow us to extract useful features of polynomial differential forms for unisolvency proof. 

In the paper the space of polynomial or polynomial differential forms without specified domain will denote the space on $\R^n$. For example, $ \mc{Q}_r\Lm^k$ is the space 
of polynomial coefficient differential forms such that the polynomials coefficients are in $\mc{Q}_r (\R^n)$. We use $\mc{P} \Lm^k$ to denote the space of all differential forms with polynomial coefficients. 

For later discussion we introduce some additional notation for cubes in a hyperspace in $\R^n$.  
For given $\mathcal{I} = \{ i_1, \ldots, i_m\} \subset \{1, \ldots, n\}$ with $i_1 < i_2 < \ldots < i_m$, consider an $m$-dimensional hyperspace $F$ in $\R^n$ determined by fixing all $x_j$ coordinates of $j \not \in \mathcal{I}$.  
If $f$ is an $m$-dimensional cube in the $m$-dimensional hyperspace $F$,
then $\Sigma_k(f)$ with $k \le m$ is defined by the set of increasing injective map from $\{ 1, \ldots, k\}$ to $\mathcal{I}$. 
For $\s \in \Sigma_k(f)$ we will use $\llbracket \s \rrbracket$ to denote the range of $\s$, i.e., 
\[ 
\llbracket \s \rrbracket = \{ \s_1, \s_2, \ldots , \s_k \} \subset \mathcal{I}, 
\]
and $\s^*$ for $\s \in \Sigma_k(f)$ is the complementary sequence in $\Sigma_{m-k}(f)$ such that
\[
\llbracket \s \rrbracket \cup \llbracket \s^* \rrbracket = \mathcal{I}.
\]
For $\s \in \Sigma_k(f)$ with $1 \le k \le m$ we will use $\s - i$ for $i \in \jump{\s}$ to denote 
the element $\tau \in \Sigma_{k-1}(f)$ such that $\jump{\tau} = \jump{\s} \setminus \{i\}$. 
For $\s \in \Sigma_k(f)$ with $0 \le k \le m-1$, $\s + j$ is defined similarly for $j \in \jump{\s^*}$.
For $i \in \jump{\s^*}$ we let $\e(i, \s) = (-1)^l$ where $l = |\{ j \in \jump{\s} \,:\, j < i \}|$.
For each $\s \in \Sigma_k(f)$ we define 
$\dd x_\s = \dd x_{\s_1} \wedge \cdots \wedge \dd x_{\s_k}$ and the set $\{ \dd x_\s \, : \, \s \in \Sigma_k(f) \, \}$ is a basis of $\Alt^k(f)$.

Recall that $\Sigma_k$ is $\Sigma_k(f)$ with $\mathcal{I} = \{ 1, \ldots, n\}$. 
A differential $k$-form $u$ on $\Omega$ then admits the representation 
\algns{
u = \sum_{\s \in \Sigma_k} u_{\s} \dd x_{\s},
}
where  the coefficients $u_{\s}$'s are scalar functions on $\Omega$. Furthermore,  the exterior derivative $\dd u$ 
can be expressed as 
\algns{
\dd u = \sum_{\s \in \Sigma_k} \sum_{i =1}^n  \pd_i u_{\s} \dd x_i \wedge \dd x_{\s} ,
}
if $\pd_i u_{\s}$ is well-defined as a function on $\Omega$. 
The Koszul operator $\kappa : \Alt^k(\R^n) \to \Alt^{k-1}(\R^n)$ is defined by contraction with the vector $x$, i.e.,
$(\kappa u)_x = x \lrcorner u_x$. As a consequence of the alternating property of $\Alt^k(\R^n)$, it therefore follows that 
$\kappa \circ \kappa = 0$.
It also follows that 
\algns{ 
\kappa (\dd x_{\s}) = \kappa ( \dd x_{\s_1} \wedge \cdots \wedge \dd x_{\s_k} ) = \sum_{i=1}^k (-1)^{i+1} x_{\s_i} \dd x_{\s_1} \wedge \cdots \wedge \widehat{\dd x_{\s_i}} \wedge \cdots \wedge \dd x_{\s_k} ,
}
where $\widehat{\dd x_{\s_i}}$ means that the term $\dd x_{\s_i}$ is omitted. 
This definition is extended to the space of differential $k$-form on $\Omega$ by linearity, i.e., 
\algns{
\kappa u = \kappa { \sum_{\s \in \Sigma_k} u_{\s} \dd x_{\s} } = \sum_{\s \in \Sigma_k} u_{\s} \kappa( \dd x_{\s} ) .
}
For future reference we note that 
\algn{ \label{eq:kappa-new-def}
\kappa \dd x_\s = \sum_{i \in \jump{\s}} \e(i, \s-i) x_i \dd x_{\s-i} .
}

If $f$ is an $(n-1)$-dimensional hyperspace of $\R^n$ obtained by fixing one coordinate, for example, 
\[ 
f = \{\, x \in \R^n \, : \, x_n = c \, \},
\]
then we can define the Koszul operator $\kappa_f$ for forms defined on $f$ by 
$(\kappa_f v)_x = (x - x^f) \lrcorner v_x $, where $x^f = (0, \ldots ,0, c)$. We note that the vector $x- x^f$ is in the tangent space 
of $f$ for $x \in f$.
Since $\tr_f ((x - x^f) \lrcorner u) = \tr_f (x - x^f) \lrcorner u$ for $x \in  f$ and $(\kappa u)_x = (x-x^f) \lrcorner u_x + x^f \lrcorner u_x$, we can conclude that
\begin{equation}\label{eq:kappa-tr}
\tr_f \kappa u = \kappa_f \tr_f u + \tr_f (x^f \lrcorner u).
\end{equation}
For a multi-index $\alpha$ of $n$ nonnegative integers, $x^{\alpha} = x_1^{\alpha_1} \cdots x_n^{\alpha_n}$. 
If $u$ is in $\mc{Q}_r\Lm^k$, the space of polynomial $k$-forms with tensor product polynomials of order $r$, then $u$ can be expressed as 
\[
u = \sum_{ \s \in \Sigma_k}  u_\s \dd x_\s, \qquad u_\s  \in \mc{Q}_r .
\]
Denoting $\mc{H}_r \Lm^k$ the space of differential $k$-forms with homogeneous polynomial coefficients of degree $r$, we also have  the identity
\algn{ \label{eq:homotopy-formula}
(\kappa \dd + \dd \kappa ) u = (r + k) u , \qquad u \in \mc{H}_r \Lm^k ,
}
cf. \cite[Section 3]{AFW06}.
Finally, throughout this paper we set $\hat{T} = [-1, 1]^n$. 

\subsection{The shape function space and the degrees of freedom of $\tilde{\mc{Q}}_r \Lm^k$}
In this subsection we define the shape function space of $\tilde{\mc{Q}}_r \Lm^k$. 
The shape functions of $\tilde{\mc{Q}}_r \Lm^k$ will be obtained by enriching the shape functions of $\mc{Q}_r^- \Lm^k$. 
There are other families of cubical finite element differential forms, cf. for example \cite{Arnold-Awanou-2014,Christiansen-Gillette-2016,Cockburn-Qiu-2014,Gillette-Kloefkorn-2016}, but these spaces are not involved in the construction of our new elements.


If $m$ is a $k$-form given by $m = p\, \dd x_{\s}$, where $\s \in \Sigma_k$ and the coefficient polynomial $p$ is a monomial, then we will call $m$ form monomial.
Before we define the shape functions of $\tilde{\mc{Q}}_r \Lm^k$ let us introduce new quantities of form monomials.
For a polynomial differential form $u_\s \dd x_\s$ we call the indices in $\jump{\s}$ (in $\jump{\s^*}$, resp.) {\it conforming indices} ({\it nonconforming indices}, resp.). For a form monomial $m = c_\alpha x^\alpha \dd x_\s \not = 0$ and $\alpha_i = s$ for a conforming (nonconforming, resp.) index $i$ of $m$, we call $i$ a conforming (nonconforming, resp.) index of degree $s$. We also define 
the conforming and nonconforming $s$-degrees of $m = c_{\alpha} x^{\alpha} \dd x_{\s}$ by 
\algns{
\cdeg_s (m) &=  | \{ i \,:\, i \in \jump{\s} \text{ and } \alpha_i = s \} |, \\
\ncdeg_s (m) &= | \{ i \,:\, i \in \jump{\s^*} \text{ and } \alpha_i = s \} |.
}
We can define $\cdeg_s(v)$ and $\ncdeg_s(v)$ if these quantities are same for any form monomial of $v$.
We remark that $\ncdeg_1(m)$ is same as the {\it linear degree} in \cite{Arnold-Awanou-2014} for a form monomial $m$. In contrast to the linear degree, 
we do {\it not} define $\cdeg_s$ and $\ncdeg_s$ for general polynomial differential forms. 

The conforming degree gives a new characterization of the shape function space of $\mc{Q}_r^- \Lm^k$ by
\algn{ \label{eq:Qr-minus}
\mc{Q}_r^- \Lm^k = \spn \{ x^\alpha \dd x_\s \in \mc{Q}_r \Lm^k \,:\, \cdeg_r (x^\alpha \dd x_\s)  = 0 \}. 
}
Defining $\mc{B}_r \Lm^k$ as
\algn{ \label{eq:Br}
\mc{B}_r \Lm^k = \spn \{  x^\alpha \dd x_\s \in \mc{Q}_r \Lm^k \,:\, \cdeg_r (x^\alpha \dd x_\s)  > 0 \} ,
}
it is easy to see that 
\algn{ \label{eq:Qr-direct-sum}
\mc{Q}_r \Lm^k = \mc{Q}_r^- \Lm^k \oplus \mc{B}_r \Lm^k.
}
We define the shape function space $\tilde{\mc{Q}}_r \Lm^k$ as
\algn{ \label{eq:Qtilde-shape}
\tilde{\mc{Q}}_r \Lm^k = \mc{Q}_r^- \Lm^k + \dd \kappa \mc{B}_r \Lm^k .
}

\begin{lemma} \label{lemma:deg-invariance}
Let $m = x^{\alpha} \dd x_\s \in \mc{P} \Lm^k$, $1 \le k \le n-1$, be given with a multi-index $\alpha$ and a positive integer $s$. Assume that $m'$ and $m''$ are given as $m' = \pd_i x^\alpha \dd x_i \wedge \dd x_\s$ with $i \not \in \jump{\s}$ and $m'' = x^\alpha x_{\s_j} \dd x_{\s_1} \wedge \cdots \wedge \widehat{\dd x_{\s_j}} \wedge \cdots \wedge \dd x_{\s_k}$ with $ j \in \jump{\s}$. Then the following identities hold:
\algn{
 \label{eq:deg-1} \ncdeg_{s} (m') &= \ncdeg_{s} (m) - \delta_{s, \alpha_i} , \\
 \label{eq:deg-2} \ncdeg_{s} (m'') &= \ncdeg_{s} (m) + \delta_{\alpha_{\s_j}, s-1},  \\ 
 \label{eq:deg-3} \ncdeg_{s+1} (m) + \cdeg_s (m) &= \ncdeg_{s+1} (m') + \cdeg_s (m') \\
\notag   & = \ncdeg_{s+1} (m'') + \cdeg_s (m'') 
}
where $\delta_{i,j}$ is the Kronecker delta.
In particular, if $m \in \mc{B}_r \Lm^k$ and $m''$ is a form monomial of $\kappa m$ with $j$ which is a nonconforming index of degree $(r+1)$, then $j$ is a conforming index of degree $r$ in $m$.
\end{lemma}
\begin{proof}
We show \eqref{eq:deg-1} and the first identity in \eqref{eq:deg-3}. 
If $\alpha_i = s$, then $i$ is a nonconforming index of degree $s$ of $m$ but not of $m'$, 
and the other nonconforming indices of degree $s$ of $m$ and $m'$ are same, so $\ncdeg_s(m') = \ncdeg_s(m) - 1$. If $s >1$, then the set of conforming indices of $m'$ with degree $s-1$ is the union of the set of conforming indices of $m$ with degree $s-1$ and $\{i\}$.
Therefore the first identity in \eqref{eq:deg-3} holds. 
If $\alpha_i \not = s$, then the sets of nonconforming indices of degree $s$ of both $m$ and $m'$ are same, so $\ncdeg_s(m') = \ncdeg_s(m)$. If $s>1$, then the sets of conforming indices of degree $s-1$ of both $m$ and $m'$ are same as well, so the first identity in \eqref{eq:deg-3} holds.

We show \eqref{eq:deg-2} and the second identity in \eqref{eq:deg-3}. 
If $\alpha_j = s-1$, then $j$ is a nonconforming index of degree $s$ of $m''$ in addition to the nonconforming indices of degree $s$ of $m$, so \eqref{eq:deg-2} holds.
If $\alpha_i \not = s-1$, then the sets of nonconforming indices of degree $s$ of $m$ and $m''$ are same, so \eqref{eq:deg-2} again holds. The second identity in \eqref{eq:deg-3} can be verified in a way similar to the argument used for the first identity in \eqref{eq:deg-3}. 
 
For the particular case, if $m \in \mc{B}_r \Lm^k$ and $m''$ is a form monomial of $\kappa m$ which has a nonconforming index of degree $(r+1)$, then the nonconforming index of $m''$ must be $\s_j$ and $\alpha_{\s_j} = r$ because $\alpha_l \le r$ for $1 \le l \le n$ and $\alpha_{\s_j}+1 = r+1$. This completes the proof.
\end{proof}

\begin{cor} \label{cor:inc-dec}
Let $m$ be a form monomial. Then for any form monomial $\tilde{m}$ in $\dd m$, $\cdeg_s(\tilde{m}) \ge \cdeg_s (m)$ for any $s \ge 1$. Similarly, for any form monomial $\tilde{m}$ in $\dd m$, $\ncdeg_s(\tilde{m}) \ge \ncdeg_s (m)$ if $s \ge 1$. 
\end{cor}
\begin{proof}
It is easy to check the assertions by Lemma~\ref{lemma:deg-invariance}.
\end{proof}
\begin{cor} \label{cor:inclusion}
The following inclusions hold:
\algn{
\label{eq:inclusion-1} \kappa \mc{Q}_r^- \Lm^k \subset \mc{Q}_r^- \Lm^{k-1}, \\
\label{eq:inclusion-2} \dd \mc{B}_r \Lm^k \subset \mc{B}_r \Lm^{k+1} \cap \dd \kappa \mc{B}_r \Lm^{k+1}, \\
\label{eq:inclusion-3} \dd \mc{Q}_r \Lm^k \subset \tilde{\mc{Q}}_r \Lm^{k+1}, \\
\label{eq:inclusion-4} \dd \kappa \mc{Q}_r \Lm^k \subset \tilde{\mc{Q}}_r \Lm^k .
}
\end{cor}
\begin{proof}
The inclusion \eqref{eq:inclusion-1} and $\dd \mc{B}_r \Lm^k \subset \mc{B}_r \Lm^{k+1}$ can be easily checked by the characterizations of $\mc{Q}_r^- \Lm^k$, $\mc{B}_r \Lm^k$ in \eqref{eq:Qr-minus}, \eqref{eq:Br}, and by Lemma~\ref{lemma:deg-invariance}. To show \eqref{eq:inclusion-2}, let $u \in \mc{H}_s \Lm^{k} \cap \mc{B}_r \Lm^{k}$ for a positive integer $s$. By \eqref{eq:homotopy-formula}, $(\dd\kappa + \kappa \dd)\dd u = (s+k) \dd u = \dd \kappa \dd u \in \dd \kappa \mc{B}_r \Lm^{k+1}$. As a consequence, $\dd u \in \mc{B}_r \Lm^{k+1} \cap \dd \kappa \mc{B}_r \Lm^{k+1}$. \eqref{eq:inclusion-3} follows from \eqref{eq:inclusion-2} and $\dd \mc{Q}_r^- \Lm^k \subset \mc{Q}_r^- \Lm^{k+1}$. Finally, \eqref{eq:inclusion-4} follows from \eqref{eq:Qr-direct-sum}, \eqref{eq:inclusion-1}, the fact $\dd \mc{Q}_r^- \Lm^k \subset \mc{Q}_r^- \Lm^{k+1}$, and \eqref{eq:inclusion-3}. 
\end{proof}

We now prove that $\tilde{\mc{Q}}_r \Lm^k$ is invariant under dilation and translation. 
\begin{lemma} \label{invariant}
If $\phi : \R^n \ra \R^n$ is a composition of dilation and translation, then 
$\phi^* \tilde{\mc{Q}}_r \Lm^k \subset \tilde{\mc{Q}}_r \Lm^k$, where $\phi^*$ is the pullback of $\phi$.
\end{lemma}
\begin{proof}

Let $\phi (x) = A x + b$ for a given invertible $n \times n$ diagonal matrix $A$ and a vector $b \in \R^n$.
To show $\phi^* \tilde{\mc{Q}}_r \Lm^k \subset \tilde{\mc{Q}}_r \Lm^k$, assume that $u \in \tilde{\mc{Q}}_r \Lm^k$ is written as $u = u^- + \dd \kappa u^+$ with $u^- \in \mc{Q}_r^- \Lm^k$ and $u^+ \in \mc{B}_r \Lm^k$. Then we have  
\algns{
\phi^* u = \phi^* u^- + \phi^* \dd \kappa u^+ = \phi^* u^- + \dd \phi^* \kappa u^+ = \phi^* u^- + \dd \kappa \phi^* u^+ + b \lrcorner \dd(\phi^* u^+ )
}
where we used $\phi^* \kappa u^+ = \kappa \phi^* u^+ + b \lrcorner (\phi^* u^+) $ in the last equality (cf. \cite[Section 3.2]{AFW06}).
We can easily check $\phi^* u^- \in \Q_r^-\Lm^k$ from the definition of $\phi^*$, and $\dd \kappa \mc{Q}_r^- \Lm^k \subset \mc{Q}_r^- \Lm^k$ from \eqref{eq:deg-2} and \eqref{eq:deg-1}. 
From \eqref{eq:Qr-direct-sum} we have 
\[
\dd \kappa \phi^* u^+ \in \dd \kappa \Q_r\Lm^k = \dd \kappa (\mc{Q}_r^- \Lm^k \oplus \mc{B}_r \Lm^k) \subset \mc{Q}_r^- \Lm^k 
+ \dd\kappa \mc{B}_r \Lm^k = \tilde{\mc{Q}}_r \Lm^k.
\]
It remains  to show 
\algn{ \label{dphi-up}
\dd (b \lrcorner (\phi^* u^+) ) \in \tilde{\mc{Q}}_r \Lm^k .
}
To see this, note that $b \lrcorner (\phi^* u^+)  \in \mc{Q}_r \Lm^{k-1}$. 
By \eqref{eq:inclusion-3} we have 
$ \dd (b \lrcorner (\phi^* u^+) ) \in 
 \tilde{\mc{Q}}_r \Lm^k$,
so \eqref{dphi-up} is proved.
\end{proof}

\begin{lemma} \label{lemma:Blm} The following hold:
\begin{itemize}
 \item[(a)] \label{B1m-1} For a form monomial $m\not = 0$ in $\mc{B}_r \Lm^k$, $\kappa m$ generates at least one form monomial whose nonconforming $(r+1)$-degree is $1$. 
 \item[(b)] The operator $\dd\kappa$ is injective on $\mc{B}_r \Lm^k$.
\end{itemize}
\end{lemma}
\begin{proof}
For $m = x^\alpha \dd x_\s \in \mc{B}_r \Lm^k$ there is at least one conforming index $\s_i$ such that $\alpha_{\s_i} = r$. 
Then $x^\alpha \kappa(\dd x_\s)$ has at least one form monomial such that $\s_i$ is a nonconforming index and its polynomial coefficient has $x_{\s_i}^{r+1}$ as a factor, so (a) is proved.

For the injectivity of $\dd\kappa$ on $\mc{B}_r \Lm^k$, it suffices to show that $\kappa$ is injective on $\mc{B}_r \Lm^k$
because $\dd$ is injective on the image of $\kappa$. To show $\kappa$ is injective on $\mc{B}_r \Lm^k$, we show that 
form monomials with positive nonconforming $(r+1)$-degree generated by $\kappa m$ for $m \in \mc{B} := \{ x^\alpha \dd x_\s \,:\, \cdeg_r (x^\alpha \dd x_\s) > 0 \}$ are distinct. More precisely, if $\kappa m$ and $\kappa \tilde{m}$ for $m, \tilde{m} \in \mc{B}$ have a same form monomial (up to $\pm 1$) whose nonconforming $(r+1)$-degree is 1, then $m= \tilde{m}$. To show it by contradiction, 
let $m = x^{\alpha} \dd x_\s$ and $\tilde{m} = x^{\tilde{\alpha}} \dd x_{\tilde{\s}}$ be two distinct elements in $\mc{B}$ and 
assume that $\kappa m$ and $\kappa \tilde{m}$ have a common form monomial with nonconforming index of degree $(r+1)$. 
From the definition of $\kappa$ and the common form monomial assumption, there exist $\s_i \in \jump{\s}$ and $\tilde{\s}_i \in \jump{\tilde{\s}}$ such that
\algn{\label{eq:kappam-identity}
x_{\s_i} x^{\alpha}  \dd x_{\s_1} \wedge \cdots \wedge \widehat{\dd x_{\s_i}} \wedge \cdots \wedge \dd x_{\s_k} = 
\pm x_{\tilde{\s}_i} x^{\tilde{\alpha}}  \dd x_{\tilde{\s}_1} \wedge \cdots \wedge \widehat{\dd x_{\tilde{\s}_i}} \wedge \cdots \wedge \dd x_{\tilde{\s}_k} .
}
Since $\s_i$ and $\tilde{\s}_i$ are the only nonconforming indices of degree $(r+1)$ by Lemma~\ref{lemma:deg-invariance}, $\s_i = \tilde{\s}_i$ and therefore $\dd x_\s = \dd x_{\tilde{\s}}$. Moreover, comparison of $x^{\alpha}$ and $x^{\tilde{\alpha}}$ leads to $\alpha = \tilde{\alpha}$, so it contradicts to $x^{\alpha} \dd x_\s \not = x^{\tilde{\alpha}} \dd x_{\tilde{\s}}$. 
\end{proof}

The following key result is a consequence of the above lemma.
\begin{theorem} For  $ 0 \leq k \leq n$ it holds that 
\algns{
\dim \tilde{\mc{Q}}_r \Lm^k = \pmat{n \\ k} (r+1)^n .
}
\end{theorem}
\begin{proof}
By Lemma~\ref{lemma:Blm}~(b), the spaces $\dd\kappa \mc{B}_r \Lm^k$ and  $\mc{B}_r \Lm^k$ have the same dimension,
so it suffices to show that $\mc{Q}_r^- \Lm^k \cap \dd \kappa \mc{B}_r \Lm^k = \{ 0 \}$.
Suppose that $0 \not = u \in \mc{B}_r \Lm^k$ and $\dd \kappa u \in \mc{Q}_r^- \Lm^k$. Every form monomial $m$ of $\dd \kappa u$ satisfies $\cdeg_r (m) \ge 1$ by \eqref{eq:Br} and Lemma~\ref{lemma:deg-invariance} because $\ncdeg_{r+1} (m) = 0$. However, it is a contradiction to the characterization of $\mc{Q}_r^- \Lm^k$ in \eqref{eq:Qr-minus}, so $u = 0$.
\end{proof}

\subsection{Degrees of freedom and unisolvence of $\tilde{\mc{Q}}_r \Lm^k$}

In this subsection we define the degrees of freedom of $\tilde{\mc{Q}}_r \Lm^k$ and prove the unisolvency for the degrees of freedom.


For the degrees of freedom we define two polynomial spaces for $\s \in \Sigma_k(f)$ for a cube $f$ included in an $m$-dimensional hyperspace by
\algns{
\mc{Q}_{r,\s} (f) = \bigotimes_{i \in \jump{\s}} \mc{P}_r(x_i) ,\qquad \mc{Q}_{r,\s^*} (f) = \bigotimes_{i \in \jump{\s^*}} \mc{P}_r(x_i) 
}
where $\mc{P}_r(x_i)$ is the space of polynomials of $x_i$ with degree less than or equal to $r$. 
The degrees of freedom of $\tilde{\mc{Q}}_r \Lm^k (\hat{T})$ is 
\algn{ \label{eq:DOF}
u \mapsto \int_f \tr_f u \wedge v, \quad f \in \lap_l(\hat{T}), k \le l \le n ,\quad  v \in \sum_{\tau \in \Sigma_{l-k}(f)} (\mc{Q}_{r-2, \tau} \otimes \mc{Q}_{r, \tau^*})(f) \dd x_{\tau} 
}
where $(\mc{Q}_{r-2, \tau} \otimes \mc{Q}_{r, \tau^*})(f)  = \mc{Q}_{r-2, \tau} (f) \otimes \mc{Q}_{r, \tau^*}(f)$. 

\begin{theorem}
The number of degrees of freedom given by \eqref{eq:DOF} is 
\algns{
\pmat{n \\ k} (r+1)^n .
}
\end{theorem}
\begin{proof}
For $\tau \in \Sigma_{l-k}(f)$ with $f \in \lap_l(\hat{T})$, $l \ge k$, 
\algns{
\dim \mc{Q}_{r-2, \tau} (f) = (r-1)^{l-k}, \qquad  \dim \mc{Q}_{r, \tau^*} (f) = (r+1)^{k}.
}
We can easily check that 
\algns{
| \lap_l (\hat{T}) | = \pmat{n \\ n-l} 2^{n-l}, \qquad | \Sigma_{l-k}(f) | = \pmat{l \\ l-k} \text{ for } f \in \lap_l(\hat{T}).
}
Therefore the number of degrees of freedom given by \eqref{eq:DOF} is 
\algns{
&\sum_{k \le l \le n} \pmat{n \\ n-l} 2^{n-l} \pmat{l \\ l-k} (r-1)^{l-k} (r+1)^k  \\
&= \frac{n ! (r+1)^k}{k! (n-k)!} \sum_{k \le l \le n} \frac{(n-k)!}{(n-l)! (l-k)!} 2^{n-l} (r-1)^{l-k} \\
&= \pmat{n \\ k} (r+1)^k \sum_{0 \le i \le n-k} \frac{(n-k)!}{(n-k-i)! i!} 2^{n-k-i} (r-1)^i \\
&= \pmat{n \\ k} (r+1)^n ,
}
so the proof is complete.
\end{proof}

The following result will be useful to reduce unisolvency proof.

\begin{theorem}[trace property]\label{thm:trace-property}
Let $f$ be a hyperspace in $\R^n$ determined by fixing one coordinate $x_i$. Then 
\algns{
\tr_f \tilde{\mc{Q}}_r \Lm^k \subset \tilde{\mc{Q}}_r \Lm^k (f) .
}
\end{theorem}
\begin{proof} 
Without loss of generality we assume that $f = \{x \in \R^n \,:\, x_n = c \}$ for some constant $c$. 
It is easy to check by definition that $\tr_f \mc{Q}_r^- \Lm^k \subset \mc{Q}_r^- \Lm^k(f)$, so 
it is enough to show that $\tr_f \dd \kappa \mc{B}_r \Lm^k \subset \tilde{\mc{Q}}_r \Lm^k (f)$.
We will show $\tr_f \dd \kappa ( x^\alpha \dd x_\s) \in \tilde{\mc{Q}}_r \Lm^k (f)$ for all $x^\alpha \dd x_\s \in \mc{B}_r \Lm^k$ below.
By \eqref{eq:kappa-tr} and the commutativity of $\tr_f$ and $\dd$, 
\algns{
\tr_f \dd \kappa ( x^\alpha \dd x_\s) = \dd \kappa_f \tr_f ( x^\alpha \dd x_\s) + \dd \tr_f (x^f \lrcorner ( x^\alpha \dd x_\s)) 
}
where $x^f$ is the vector field $(0, \ldots, 0, c)$ on $\R^n$. 

If $n \in \jump{\s}$, then $\tr_f (x^\alpha \dd x_\s) = 0$ and
$\tr_f (x^f \lrcorner (x^\alpha \dd x_\s) ) = (c x^\alpha |_{x_n = c}) \dd x_{\tilde{\s}} \in \mc{Q}_r \Lm^{k-1}(f)$ where $\dd x_{\tilde{\s}}$ is defined by $\dd x_\s = \dd x_{\tilde{\s}} \wedge \dd x_n$. From \eqref{eq:inclusion-3} 
we can conclude that $\tr_f \dd \kappa ( x^\alpha \dd x_\s) \in \tilde{\mc{Q}}_r \Lm^k(f)$. 
If $n \not \in \jump{\s}$, then $x^f \lrcorner (x^\alpha \dd x_\s) = 0$. Since $\tr_f (x^\alpha \dd x_\s) = x^\alpha|_{x_n = c} \dd x_\s \in \mc{Q}_r \Lm^k(f)$, the conclusion follows from \eqref{eq:inclusion-4}.
\end{proof}

Before we start the unisolvence proof we need a lemma and auxiliary definitions. 

\begin{lemma} \label{lemma:eps-identity}
Suppose that $\s, \tilde{\s} \in \Sigma_k$, $1 \le k \le n-1$, $\s \not = \tilde{\s}$ satisfy $\s + {i} = \tilde{\s} + \tilde{i}$ for some $i \in \jump{\tilde{\s}}$, $\tilde{i}\in \jump{\s}$. Then 
\algns{
\e(i, {\s} ) \e(i, \tilde{\s}-i) - \e(\tilde{i}, \tilde{\s} ) \e(\tilde{i}, \s - \tilde{i}) \not{=} 0. 
}
\end{lemma}
\begin{proof}
We first prove it under the assumption $i > \tilde{i}$. Let 
\algns{
a &= | \{ l \,:\, l < \tilde{i} , l \in \jump{\s} \cap \jump{\tilde{\s}} \} |, \\
b &= | \{ l \,:\, \tilde{i} < l < i, l \in \jump{\s} \cap \jump{\tilde{\s}} \} |, \\
c &= | \{ l \,:\, i < l , l \in \jump{\s} \cap \jump{\tilde{\s}} \} |. 
}
Then one can check 
\gats{
\e(i, \s) = (-1)^{a+b+1}, \quad \e(\tilde{i}, \tilde{\s}) = (-1)^{a}, \\
\e(i, \tilde{\s} - i) = (-1)^{a+b}, \quad \e(\tilde{i}, \s - \tilde{i}) = (-1)^a, 
}
so the assertion follows. If $i < \tilde{i}$, then we set
\algns{
a &= | \{ l \,:\, l < i , l \in \jump{\s} \cap \jump{\tilde{\s}} \} |, \\
b &= | \{ l \,:\, i < l < \tilde{i}, l \in \jump{\s} \cap \jump{\tilde{\s}} \} |, \\
c &= | \{ l \,:\, \tilde{i} < l , l \in \jump{\s} \cap \jump{\tilde{\s}} \} |,
}
and one can check that
\gats{
\e(i, \s) = (-1)^{a}, \quad \e(\tilde{i}, \tilde{\s}) = (-1)^{a+b+1}, \\
\e(i, \tilde{\s} - i) = (-1)^{a}, \quad \e(\tilde{i}, \s - \tilde{i}) = (-1)^{a+b}.
}
The proof is complete.
\end{proof}

We define $\mc{D}_{r,l} \Lm^k$ as 
\algns{
\mc{D}_{r,l} := \{ u \in \mc{P} \Lm^k \,:\, \ncdeg_{r+1} (m) + \cdeg_r (m) = l \text{ for every form monomial } m \text{ in } u \} .
}
By Lemma~\ref{lemma:deg-invariance} $\dd \kappa$ maps $\mc{D}_{r,l} \Lm^k$ into itself. 
Considering the decomposition of $\mc{B}_r \Lm^k$ 
\algns{
\mc{B}_r \Lm^k = \bigoplus_{1 \le l \le k} \bigoplus_{r \le s \le nr} \mc{H}_s \Lm^k \cap \mc{D}_{r,l} \Lm^k \cap \mc{B}_r \Lm^k 
}
we have a decomposition of $\dd\kappa \mc{B}_r \Lm^k$
\algn{ \label{eq:dkappa-decomp}
\dd \kappa \mc{B}_r \Lm^k = \bigoplus_{1 \le l \le k} \bigoplus_{r \le s \le nr} \mc{H}_s \Lm^k \cap \mc{D}_{r,l} \Lm^k \cap \dd \kappa \mc{B}_r \Lm^k .
}

Let us recall the degrees of freedom of $\mc{Q}_r^- \Lm^k(\hat{T})$ with vanishing trace in \cite{Arnold-Boffi-Bonizzoni-2015}. If $u \in \mc{Q}_r^- \Lm^k(\hat{T})$ and $\tr_f u = 0$ for all $f \in \lap_{n-1}(\hat{T})$, then 
\algn{ \label{eq:Qr-minus-DOF}
\int_{\hat{T}} u \wedge v = 0, \qquad \forall v \in \mc{Q}_{r-1}^- \Lm^{n-k}(\hat{T}) 
}
implies that $u=0$. 

We are now ready to prove unisolvency of $\tilde{\mc{Q}}_r \Lm^k (\hat{T})$, $r \ge 1$, with the degrees of freedom \eqref{eq:DOF}. 

\begin{prop}[unisolvence with vanishing trace assumption] \label{thm:reduced-unisolvence}
Suppose that $u \in \tilde{\mc{Q}}_r \Lm^k(\hat{T})$, $r \ge 1$, and $\tr_f u = 0$ for all $f \in \lap_{n-1}(\hat{T})$.
If 
\algn{ \label{eq:reduced-DOF}
\int_{\hat{T}} u \wedge v = 0 \qquad \forall v \in \sum_{\tau \in \Sigma_{n-k}(\hat{T})} (\mc{Q}_{r-2, \tau} \otimes \mc{Q}_{r, \tau^*}) (\hat{T}) \dd x_{\tau}, 
}
then $u = 0$. Here we accept $\mc{Q}_{-1} = \emptyset$ for convention.
\end{prop}
\begin{proof}
If $k=0$ or $k=n$, then $\tilde{\mc{Q}}_r \Lm^k (\hat{T}) = \mc{Q}_r \Lm^k (\hat{T})$ and \eqref{eq:reduced-DOF} is a standard set of degrees of freedom for the shape functions with vanishing trace, so there is nothing to prove.

Assume that $0 < k < n$, and let $u = \sum_{\s \in \Sigma_k(\hat{T})} u_\s \dd x_\s \in \tilde{\mc{Q}}_r \Lm^k(\hat{T})$ be a shape function with vanishing trace.
From the vanishing trace assumption $\tr_f u = 0$ for all $f \in \lap_{n-1}(\hat{T})$, $u_\s$ vanishes on all faces $f \in \lap_{n-1}(\hat{T})$ determined by $x_i = \pm 1$ for any $i \in \jump{\s}$. Therefore $u_\s = b_{\s^*} \tilde{u}_{\s}$ with $b_{\s^*} := \prod_{l \in \jump{\s^*}} (1-x_l^2)$
for all $\s \in \Sigma_k(\hat{T})$. In the degree properties \eqref{eq:deg-1}, \eqref{eq:deg-2}, \eqref{eq:deg-3}, 
the coefficients of all form monomials in $\tilde{\mc{Q}}_r \Lm^k(\hat{T})$ have at most one variable of degree $(r+1)$. From this observation and \eqref{eq:reduced-DOF}, $u$ has a form
\algn{ \label{eq:u-reduced}
u = \sum_{\s \in \Sigma_k(\hat{T})} b_{\s^*} \LRp{\sum_{i \in \jump{\s^*}} L_{r-1}^w(x_i) p_{\s,i} } \dd x_\s
} 
where $L_s^w(t)$ is the monic Legendre polynomial of degree $s$ on $[-1, 1]$ with weight $(1-t^2)$, 
and $p_{\s,i}$ is a polynomial in $(\mc{Q}_{r,\s}  \otimes \mc{Q}_{r-2, \s^*}) (\hat{T})$ which is independent of $x_i$.
Recall the decomposition \eqref{eq:dkappa-decomp} and let 
\algns{
u = u_0 + \sum_{l,s} u_{l,s} \in \mc{Q}_r^- \Lm^k (\hat{T}) \oplus \bigoplus_{1 \le l \le k} \bigoplus_{r\le s \le nr} \mc{H}_s \Lm^k (\hat{T}) \cap \mc{D}_{r,l} \Lm^k (\hat{T}) \cap \dd \kappa \mc{B}_r \Lm^k (\hat{T}).
}
Let $l_0$ be the largest index $l$ such that $u_{l,s} \not = 0$ for some $s$, and let $s_0$ be the largest index $s$ such that $u_{l_0, s} \not = 0$. 
In the proof below, we will show $u_{l_0, s_0} = 0$. By induction, this implies that $u = u_0 \in \mc{Q}_r^- \Lm^k(\hat{T})$, and therefore $u=0$ by \eqref{eq:reduced-DOF}, \eqref{eq:Qr-minus-DOF}, and the inclusion 
\algns{
\mc{Q}_{r-1}^- \Lm^{n-k}(\hat{T}) \subset \sum_{\tau \in \Sigma_{n-k}(\hat{T})} (\mc{Q}_{r-2, \tau} \otimes \mc{Q}_{r, \tau^*}) (\hat{T}) \dd x_{\tau} .
}

We now start the proof of $u_{l_0, s_0} = 0$. In the form \eqref{eq:u-reduced}, if $p_0$ is a monomial of $p_{\s,i}$ for fixed $\s$, then $\cdeg_r (p_0 \dd x_\s) < l_0$ because any form monomial $m$ containing the factors $x_i^{r+1}$ and $p_0$ in its polynomial coefficient,
which comes from the expansion of the coefficient of $\dd x_{\s}$ in \eqref{eq:u-reduced}, 
satisfies both of $\ncdeg_{r+1}(m) = 1$ and $\ncdeg_{r+1}(m) + \cdeg_r (m) \le l_0$.
Note that it also implies that $u_{l_0, s_0}$ does not have any form monomials with conforming $r$-degree $l_0$. 

Let $q_{\s,i}$ be the homogeneous polynomial of $p_{\s,i}$ for each $\s$ and $i \in \jump{\s^*}$ which contribute to consist $u_{l_0, s_0}$, i.e., 
\algn{ \label{eq:u-max}
u_{l_0, s_0} = \sum_{\s \in \Sigma_k(\hat{T})} \prod_{l \in \jump{\s^*}} x_l^2 \LRp{\sum_{i \in \jump{\s^*}} x_i^{r-1} q_{\s,i} } \dd x_\s
}
and $\cdeg_r (q_{\s,i} \dd x_\s ) = l_0 - 1$ by the previous observation. Since $u_{l_0, s_0} \in \dd \kappa \mc{B}_r \Lm^k (\hat{T})$ by definition, $\dd u_{l_0, s_0} = 0$. 

We now show $u_{l_0, s_0} = 0$ for $l_0 = 1$. If $l_0 = 1$, then $q_{\s,i} \in (\mc{Q}_{r-1, \s} \otimes \mc{Q}_{r-2, \s^*}) (\hat{T})$ because $\cdeg_r (q_{\s,i} \dd x_\s) < l_0$. We claim that $q_{\s,i} = 0$ for all $\s$ and $i$. To prove it we show that the form monomials in $\dd u_{l_0, s_0}$ with conforming $r$-degree 1, are all distinct.
Note that such form monomials are from
\algn{ \label{eq:tmp-1}
(r+1) \prod_{l \in \jump{\s^*}} x_l^2 x_i^{r-2} q_{\s,i} \dd x_i \wedge \dd x_\s
}
for $i \in \jump{\s^*}$. For fixed $\s$ the form monomials in this formula are all distinct for different $i$'s. Further, if we assume that there is $\tilde{\s} \in \Sigma_k(\hat{T})$ with $\tilde{\s} \not = \s$ and $\tilde{i} \in \jump{\tilde{\s}^*}$ such that the form monomials in 
\algns{
(r+1) \prod_{l \in \jump{\tilde{\s}^*}} x_l^2 x_i^{r-2} q_{\tilde{\s},\tilde{i}} \dd x_{\tilde{i}} \wedge \dd x_{\tilde{\s}}
}
are not linearly independent with the ones in \eqref{eq:tmp-1}, 
then it leads to a contradiction because $i$ and $\tilde{i}$ are the only conforming indices of degree $r$ in these form monomials, and therefore $i=\tilde{i}$, $\s = \tilde{\s}$. 
Thus, all form monomials in $\dd u_{l_0, s_0}$ with conforming $r$-degree 1 are distinct and have the form \eqref{eq:tmp-1}. From $\dd u_{l_0, s_0} = 0$, $q_{\s, i} = 0$ for all $\s \in \Sigma_k (\hat{T})$ and $i \in \jump{\s^*}$, and therefore $u_{l_0, s_0} = 0$ by \eqref{eq:u-max}. 

If $l_0 >1$, 
then we consider the expression of $q_{\s,i}$ as a sum of homogeneous polynomials
\algn{ \label{eq:q_si}
q_{\s,i} = \sum_{\tau \subset \jump{\s} , |\tau| = l_0 -1} \prod_{j \in \tau} x_j^r q_{\s,i,\tau} 
}
in which $q_{\s,i,\tau} \in (\mc{Q}_{r-1, \s}  \otimes \mc{Q}_{r-2, \s^*}) (\hat{T})$ is independent of the variables $x_i$ and $x_j$'s for $j \in \tau$.
Here $q_{\s,i, \tau}$ can be vanishing for some $\tau \subset \jump{\s}$. In the discussion below, we will make a system of equations with all possibly involving $q_{\s, i, \tau}$'s as its unknown. The equations are obtained from $\kappa u_{l_0, s_0}$ and $\dd u_{l_0, s_0}$, and the right-hand sides are zero. We then show that the system is well-posed, so all $q_{\s, i, \tau}$'s vanish. 

To derive equations from $\kappa u_{l_0, s_0}$, note that $\kappa u_{l_0, s_0} \in \kappa \dd \kappa \mc{B}_r\Lm^k (\hat{T}) = \kappa \mc{B}_r \Lm^k (\hat{T})$, so the nonconforming $(r+1)$-degree of form monomials in $\kappa u_{l_0, s_0}$ is at most 1 by Lemma~\ref{lemma:deg-invariance}. In the formal expression of $\kappa u_{l_0, s_0}$ using \eqref{eq:u-max} and \eqref{eq:q_si}, for fixed $\s$, the form monomials which have nonconforming $(r+1)$-degree 2 are obtained by
\algn{ \label{eq:kappa-uls}
\prod_{l \in \jump{\s^*}} x_l^2  x_i^{r-1}  \prod_{l \in \tau} x_l^r x_j \e(j,\s-j) q_{\s,i,\tau} \dd x_{\s-j}
}
for $\s \in \Sigma_k(\hat{T})$, $i \in \jump{\s^*}$, $\tau \subset \jump{\s}$ with $|\tau| = l_0 -1$, and $j \in \tau$.
We claim that all of these terms are linearly independent for fixed $\s\in \Sigma_k(\hat{T})$. To see it, assume that the form \eqref{eq:kappa-uls} with $(i, \tau, j)$
and $(i', \tau', j')$ are linearly dependent for $i' \in \jump{\s^*}$, $\tau' \subset \jump{\s}$, $j'\in\tau'$.
Comparison of the nonconforming indices with degree $(r+1)$ gives $i = i'$, $j = j'$ or $i = j'$, $j = i'$.
If $i = i'$ and $j=j'$, then $\tau = \tau'$ from the comparison of conforming indices of degree $r$. 
If $i = j'$ and $j = i'$, then $\dd x_{\s-j} \not = \dd x_{\s-j'}$, so they cannot be linearly dependent. Therefore 
the terms of the form \eqref{eq:kappa-uls} are all distinct for fixed $\s$. 

We now assume that \eqref{eq:kappa-uls} with $(\s, i, \tau, j)$ and $(\tilde{\s}, \tilde{i}, \tilde{\tau}, \tilde{j})$ are linearly dependent for $\tilde{\s} \not =\s$. Comparing the nonconforming indices of degree $(r+1)$, $i=\tilde{i}$, $j = \tilde{j}$ or $i = \tilde{j}$, $j = \tilde{i}$. However, if $i=\tilde{i}$ and $j = \tilde{j}$, then $\dd x_{\s-j} = \dd x_{\tilde{\s} - \tilde{j}}$ implies $\s = \tilde{\s}$ which is a contradiction, so $i = \tilde{j}$ and $j = \tilde{i}$. Regarding this and comparing the conforming indices of degree $r$, linear dependence occurs only when 
\algn{ \label{eq:comp-quad}
i = \tilde{j}, j = \tilde{i}, \quad 
\tau \cup \{i\} = \tilde{\tau} \cup \{\tilde{i}\}, \quad \s - \tilde{i} = \tilde{\s} - i .
}
As a consequence, for fixed $(\s, i, \tau, j)$ with $i \in \jump{\s}$ and $j \in \tau$, there is a unique quadruple $(\tilde{\s}, \tilde{i}, \tilde{\tau}, \tilde{j})$ determined by \eqref{eq:comp-quad} which may generate a linearly dependent polynomial differential form in the form of \eqref{eq:kappa-uls}. Since 
all form monomials which have nonconforming $(r+1)$-degree 2 are vanishing, 
\algn{ \label{eq:q-eq1}
\e(\tilde{i}, \s - \tilde{i}) q_{\s, i, \tau} + \e(i, \tilde{\s} - i) q_{\tilde{\s}, \tilde{i}, \tilde{\tau}} = 0
}
for the quadruples $(\s, i, \tau, j)$ and $(\tilde{\s}, \tilde{i}, \tilde{\tau}, \tilde{j})$ satisfying \eqref{eq:comp-quad}.

We now consider the expressions of form monomials with conforming $r$-degree $l_0$ in $\dd u_{l_0, s_0}$. From \eqref{eq:u-max} and \eqref{eq:q_si}, they have a form
\algn{ \label{eq:d-uls}
(r+1) x_i^{r-2}\prod_{l \in \jump{\s^*}} x_l^2 \prod_{l \in \tau} x_l^r \e(i, \s) q_{\s, i, \tau} \dd x_{\s + i} .
}
By checking the differential form component and the conforming indices of degree $r$ one can check that all of these terms are distinct over $i \in \jump{\s^*}$ and $\tau \subset \jump{\s}$ with $|\tau| = l_0 - 1$ for fixed $\s$. 

Considering $(\s, i, \tau)$ and $(\tilde{\s}, \tilde{i}, \tilde{\tau})$ generating linearly dependent differential forms of the form \eqref{eq:d-uls}, a direct comparison gives 
\algns{
\s + i = \tilde{\s} + \tilde{i}, \quad \tau \cup \{i\} = \tilde{\tau} \cup \{\tilde{i}\} 
}
which is exactly same as \eqref{eq:comp-quad}. The identity $\dd u_{l_0, s_0} = 0$ with \eqref{eq:kappa-uls} gives a new equation
\algn{ \label{eq:q-eq2}
\e(i, \s) q_{\s, i, \tau} + \e(\tilde{i}, \tilde{\s}) q_{\tilde{\s}, \tilde{i}, \tilde{\tau}} = 0.
}
By Lemma~\ref{lemma:eps-identity}, \eqref{eq:q-eq1}, and \eqref{eq:q-eq2}, it follows that $q_{\s, i, \tau} = q_{\tilde{\s}, \tilde{i}, \tilde{\tau}} = 0$. Since this is true for any $\s$, $i$, $\tau$, $u_{l_0,s_0}  = 0$ and this completes the proof.
\end{proof}

\begin{theorem}[unisolvence] \label{thm:unisolvence}
Let $u \in \tilde{\mc{Q}}_r \Lm^k(\hat{T})$. If 
\algn{ \label{eq:DOF-all}
\int_f \tr_f u \wedge v = 0 \qquad \forall v \in \sum_{\tau \in \Sigma_{l-k}(f)} (\mc{Q}_{r-2, \tau} \otimes \mc{Q}_{r,\tau^*}) (f) \dd x_{\tau},
}
for all $f \in \lap_l(\hat{T})$, $l \ge k$, then $u = 0$.
\end{theorem}
\begin{proof}
If $l = k$, then $\tr_f u \in \tilde{\mc{Q}}_r \Lm^k(f) = \mc{Q}_r \Lm^k(f)$ for $f \in \lap_k(\hat{T})$ by Theorem~\ref{thm:trace-property}, 
and \eqref{eq:DOF-all} gives a standard set of degrees of freedom for $\mc{Q}_r \Lm^k(f)$, so $\tr_f u = 0$.
Suppose that $\tr_f u = 0$ for all $f \in \lap_l(\hat{T})$ for some $l \ge k$. For any given $\tilde{f} \in \lap_{l+1}(\hat{T})$, $\tr_{\tilde{f}} u \in \tilde{\mc{Q}}_r \Lm^k(\tilde{f})$ and $\tr_f \tr_{\tilde{f}} u = \tr_{f} u =0$ for all $f \in \lap_l(\tilde{f})$ by the assumption. By Proposition~\ref{thm:reduced-unisolvence} and the degrees of freedom \eqref{eq:DOF-all}, $\tr_{\tilde{f}} u = 0$. By induction, we can show that $\tr_f u = 0$ for all $f \in \lap_{n-1}(\hat{T})$, so $u= 0$ by Proposition~\ref{thm:reduced-unisolvence}.
\end{proof}

\subsection{Nodal tensor product degrees of freedom }

In this subsection we show that $\tilde{\mc{Q}}_r \Lm^k(\hat{T})$ has a set of degrees of freedom given by evaluating nodal values at
a set of points in $\hat{T}$. This alternative set of degrees of freedom will be used to define numerical methods with local coderivatives in the next section.

The Gauss--Lobatto quadrature rule with $r+1 (r \ge 1)$ points uses $r-1$ interior points
and two end points of $I = [-1,1]$ with positive weights and gives exact integration of polynomials of order $2r-1$.
Let $\{\xi^j \}_{j=0}^{r}$ be the quadrature points of the Gauss-Lobatto quadrature on $I$ and $\{ \lambda^j\}_{j=0}^r$ with $\lambda^j >0$ be the weights at the points. 
We can define a quadrature rule on $\hat{T}$ by taking tensor products of the Gauss-Lobatto 
quadrature nodes and weights, and we discuss the formal expressions for the tensor product quadrature rules below. 

For $\s \in \Sigma_k(\hat{T})$ let $N_{\s,r}$ be the set of points in $\bigotimes_{1 \le i \le k} I(x_{\s_i})$ defined by
\algns{
N_{\s,r} = \{ (x_{\s_1}, \ldots, x_{\s_k}) = (\xi^{j_1}, \ldots , \xi^{j_k}) \,:\, 0 \le j_l \le r, 1 \le l \le k \}.
}
$N_{\s^*,r}$ is defined similarly as 
\algns{
N_{\s^*,r} = \{ (x_{\s^*_1}, \ldots, x_{\s^*_{n-k}}) = (\xi^{j_1}, \ldots , \xi^{j_{n-k}}) \,:\, 0 \le j_l \le r, 1 \le l \le n-k \},
}
the set of points in the $(n-k)$-dimensional cube $\bigotimes_{1 \le i \le n-k} I(x_{\s^*_i})$. Therefore the tensor product $N_r := N_{\s,r} \otimes N_{\s^*, r}$ is the set given by the tensor product of Gauss-Lobatto quadrature points on $\hat{T}$.
Letting $\Bbb{N}$ be the set of nonnegative integers, we will use $\xi_\s^{\bs{i}}$ and $\xi_{\s^*}^{\bs{j}}$ with multi-indices $\bs{i} \in \Bbb{N}^k$ and $\bs{j} \in \Bbb{N}^{n-k}$ to denote the points in $N_{\s, r}$ and $N_{\s^*, r}$, respectively. We also use $\lambda_\s^{\bs{j}}$ and $\lambda_{\s^*}^{\bs{j}}$ to denote the corresponding tensor-product weights of the Gauss-Lobatto quadrature.

For a continuous function $v$ on $\hat{T}$, we define $R_{r}(v)$, $E_{r, \xi_{\s}^{\bs{i}} } (v)$, $E_r(v)$ as 
\algns{
R_{r}(v) &= \LRp{ v(\xi_\s^{\bs{i}} \otimes \xi_{\s^*}^{\bs{j}} ) }_{\xi_\s^{\bs{i}} \otimes \xi_{\s^*}^{\bs{j}} \in N_{r}} \in \R^{(r+1)^n}, \\
E_{r, \xi_{\s}^{\bs{i}} } (v) &= \sum_{\xi_{\s^*}^{\bs{j}} \in N_{r,\s^*} }  \lambda_{\s^*}^{\bs{j}} v(\xi_\s^{\bs{i}} \otimes \xi_{\s^*}^{\bs{j}} ), \\
E_{r}(v) &= \sum_{\xi_{\s}^{\bs{i}} \in N_{r, \s} } \lambda_\s^{\bs{i}} E_{r, \xi_{\s}^{\bs{i}} } (v) .
}
Similarly, for polynomial differential forms $u = u_\s \dd x_\s \in \mc{P} \Lm^k (\hat{T})$ 
we define $R_{r} (u)$ as the element in $\R^M \otimes \Lm^k$ with $M={\pmat{n \\ k} (r+1)^n}$, which consists of $R_{r}(u_\s) \otimes \dd x_\s$ for $\s \in \Sigma_k(\hat{T})$.

\begin{lemma} \label{lemma:reduced-nodal-DOF}
For $v = \sum_{\s \in \Sigma_k(\hat{T})} v_{\s} \dd x_{\s}$ suppose that a polynomial $v_{\s}$ has a form 
with 
\algns{
v_{\s} = b_{\s^*} \tilde{v}_{\s} , \qquad b_{\s^*} := \prod_{l \in \jump{\s^*}} (1-x_l^2),  \quad 
\tilde{v}_{\s} \in (\mc{Q}_{r-1,\s^*} \otimes \mc{Q}_{r, \s}) (\hat{T}) .
}
Suppose also that nonconforming $(r+1)$-degree of every form monomial in $v_{\s} \dd x_\s$ is at most 1 for all $\s \in \Sigma_k(\hat{T})$. 
If $R_{r}(v_{\s}) = 0$ all $\s \in \Sigma_k(\hat{T})$, then $v = 0$.
\end{lemma}

\begin{proof}
If we rewrite $\tilde{v}_{\s}$ with the weighted Legendre polynomial $L_j^w$'s, then
the assumption on nonconforming $(r+1)$-degree leads us to have 
\algn{ \label{eq:v-tilde-form}
\tilde{v}_{\s} = \sum_{i \in \jump{\s^*}} L_{r-1}^w(x_i) p_{\s,i} + \tilde{v}_{\s,0}
}
with $p_{\s,i} \in (\mc{Q}_{r,\s} \otimes \mc{Q}_{r-2, \s^*}) (\hat{T})$ independent of $x_i$, 
in which $\tilde{v}_{\s,0}$ is a sum of polynomials of the form
\algn{ \label{eq:v0-tilde}
\prod_{i \in \jump{\s^*}} L_{d_i}^w (x_i) q_{\bs{d}} (x_{\s})
}
where $\bs{d} = (d_{\s^*_1}, \ldots, d_{\s^*_{n-k}}) \in \Bbb{N}^{n-k}$ with 
$\max_{i \in \jump{\s^*}} \{ d_i \} \le r-2$ and $q_{\bs{d}} \in \mc{Q}_{r, \s} (\hat{T})$.

We first to show that $\tilde{v}_{\s,0} = 0$. 
Let $0 \not = \psi = \prod_{i \in \jump{\s^*}} L_{d_i'}^w (x_i) \in \mc{Q}_{r-2, \s^*}(\hat{T})$ with $d_i'$'s satisfying $0 \le d_i' \le r-2$, $i \in \jump{\s^*}$.
We claim that 
\algns{
E_{r,\xi_{\s}^{\bs{i}} } \LRp{ b_{\s^*} L_{r-1}^w (x_i) p_{\s,i} \psi } = 0
}
for all $i \in \jump{\s^*}$ and $\xi_{\s}^{\bs{i}} \in N_{r,\s}$. To see it, note that the quadrature along $x_i$ coordinate can be replaced by integration with $x_i$ variable on $[-1, 1]$ 
because the degree of $x_i$ variable is $r + 1+ d_i' \le 2r -1$
and $L_{r-1}^w(x_i)$ is orthogonal to $L_{d_i'}^w(x_i)$ with $(1-x_i^2)$ weight. Therefore we have
\algns{
E_{r,\xi_{\s}^{\bs{i}} }(v \phi) = E_{r,\xi_{\s}^{\bs{i}} } \LRp{ b_{\s^*} \tilde{v}_{\s,0} \psi } .
}
If we consider the expression of $\tilde{v}_{\s,0}$ in \eqref{eq:v0-tilde}, 
a completely analogous argument using orthogonality of Legendre polynomials gives 
\algns{
E_{r,\xi_{\s}^{\bs{i}} }(v \phi) = E_{r,\xi_{\s}^{\bs{i}} } \LRp{ b_{\s^*} \psi^2 (x_{\s^*}) q_{\bs{d}'}(x_\s)  } .
}
where $\bs{d}' = (d_{\s_1^*}', \ldots, d_{\s_{n-k}^*}')$.
Note that this result is obtained without using the assumption $R_r(v_{\s}) = 0$.  
Since $R_{r}(v_{\s}) = 0$, the above quantity vanishes. By the definition of $E_{r,\xi_{\s}^{\bs{i} } }$, we have either $b_{\s^*} (\xi_{\s^*}^{\bs{j}} ) \psi(\xi_{\s^*}^{\bs{j}})=0$ for all $\xi_{\s^*}^{\bs{j}} \in N_{r, \s^*}$ or $q_{\bs{d}'} (\xi_{\s}^{\bs{i}}) = 0$. However, the first case implies that $b_{\s^*} \psi  = 0$ because the nodal value evaluations at the points in $N_{r, \s^*}$ is already a set of degrees of freedom for $\mc{Q}_{r, \s^*} (\hat{T})$, and it leads to a contradiction to $\psi \not = 0$. Therefore $q_{\bs{d}}$ vanishes at $\xi_{\s}^{\bs{i}}$, and we can show that $q_{\bs{d}}$ vanishes at any point in $N_{r,\s}$ with the same argument. Recall that $q_{\bs{d}'} \in \mc{Q}_{r,\s} (\hat{T})$, so these vanishing conditions of $q_{\bs{d}'}$ implies that $q_{\bs{d}'} = 0$. Finally, this holds for any $\bs{d}'$, and therefore $\tilde{v}_{\s,0} = 0$.

To show $v = 0$, we notice that $v_{\s}$ with \eqref{eq:v-tilde-form} is exactly the form of $u_{\s}$ in \eqref{eq:u-reduced}, so the same argument in the unisolvency proof can be used to show $v = 0$. 
\end{proof}

\begin{theorem}
Suppose that $v \in \tilde{\mc{Q}}_r \Lm^k (\hat{T})$ and $R_{r} (v) = 0$  holds. Then $v = 0$. 
\end{theorem}
\begin{proof}
Note that $\tr_f v \in \tilde{\mc{Q}}_r \Lm^k(f) = \mc{Q}_r \Lm^k(f)$ for $f \in \lap_k(\hat{T})$. Since the restriction of $R_r$ on $f$ becomes a set of quadrature degrees of freedom of  $\mc{Q}_r \Lm^k(f)$, $\tr_f v = 0$ for all $f \in \lap_k(\hat{T})$ holds. For $\tr_g v$ with $g \in \lap_{k+1}(\hat{T})$, all traces of $\tr_g v$ on $k$-dimensional subcubes are vanishing, so the assumption of Lemma~\ref{lemma:reduced-nodal-DOF} is satisfied for $\tr_g v$. Applying Lemma~\ref{lemma:reduced-nodal-DOF} with the restriction of $R_r$ on $g$, one can conclude that $\tr_g v = 0$ for any $g \in \lap_{k+1}(\hat{T})$. We can continue this argument for any $g \in \lap_l(\hat{T})$, $l \ge k+1$, by induction, so the conclusion follows. 
\end{proof}

\section{Numerical methods with local coderivatives}\label{local-c}

We construct numerical methods with local coderivatives using $\tilde{\mc{Q}}_r \Lm^k (\mc{T}_h)$.
For this we need a modified bilinear form $\LRa{\cdot, \cdot}_h$ in {\bf (A)},
and the auxiliary spaces $\tilde{V}_h^{k-1}$, $W_h^{k-1}$ and the map $\Pi_h$ in {\bf (B)}.
The conditions {\bf (A)}, {\bf (B)} are stated with index $k-1$ but for simplicity we will check 
the conditions in this section for $V_h^k$, $\tilde{V}_h^k$, $W_h^k$ which we choose as 
\algn{ \label{eq:V-spaces}
V_h^k = \tilde{\mc{Q}}_r \Lm^k(\mc{T}_h), \quad \tilde{V}_h^k = {\mc{Q}}_r^- \Lm^k(\mc{T}_h), \quad W_h^k = {\mc{Q}}_{r-1}^d \Lm^k(\mc{T}_h)
}
where $\mc{Q}_{r-1}^d \Lm^k(\mc{T}_h)$ is the space of discontinuous polynomial differential forms.
We will show that the finite elements $\tilde{\mc{Q}}_r \Lm^k (\mc{T}_h)$ and $\mc{Q}_r^- \Lm^k (\mc{T}_h)$ satisfy the conditions {\bf (A)}, {\bf (B)}. 
In addition to the conditions in {\bf (A)} and {\bf (B)}, for local coderivatives, we also need to show that $\LRa{\cdot, \cdot}_h$ on $\tilde{\mc{Q}}_r \Lm^k(\mc{T}_h)$ can give a block diagonal matrix with appropriate choice of global DOFs.

We have shown that $\tilde{\mc{Q}}_r \Lm^k({\hat{T}})$ has a set of DOFs determined by evaluations at nodal points. 
For $T \in \mathcal{T}_h$ we can define the evaluation operator $E_r^T$ for continuous functions on $T$ with the scaled Gauss-Lobatto quadrature rules on $T$. 

For $u, v \in \tilde{\mc{Q}}_r \Lm^k(T)$ with expressions 
$u = \sum_{\s \in \Sigma_k(T)} u_{\s} \dd x_{\s}$, $v = \sum_{\s \in \Sigma_k(T)} v_{\s} \dd x_{\s}$,  we define $\LRa{u, v}_{h,T}$ by
\begin{equation} \label{quadrature-c}
\LRa{u, v }_{h,T} = |T| \sum_{\s \in \Sigma_k(T)} E_r^T (u_{\s} v_{\s} ) .
\end{equation}
It is easy to see that $\LRa{\cdot, \cdot }_{h,T}$ is an inner product on $\tilde{\mc{Q}}_r \Lm^k(T)$ and the norm defined by this inner product is equivalent to the $L^2$ norm with constants independent on the scaling of $T$. We can define the inner product $\LRa{\cdot, \cdot}_h$ on $\tilde{\mc{Q}}_r \Lm^k(\mathcal{T}_h)$ by 
\algn{ \label{quadrature-all}
\LRa{u, v}_h = \sum_{T \in \mathcal{T}_h} \LRa{ u|_T, v|_T}_{h,T}
}
for $u, v \in \tilde{\mc{Q}}_r \Lm^k(\mathcal{T}_h)$, and the norm $\| \cdot \|_h$ defined by this inner product is equivalent to the $L^2$ norm with constants independent of $h$. Therefore {\bf (A)} is satisfied.


We now verify {\bf (B)}, but with $k$-forms instead of $(k-1)$-forms for notational convenience. Recall that we already defined $\tilde{V}_h^k$ and $W_h^k$ in
\eqref{eq:V-spaces}, so verification of {\bf (B)} consists of the following three steps:
\begin{itemize}
  \item[Step 1.]  Prove  \eqref{eq:B-cond} 
  \item[Step 2.] Define  $\Pi_h$  satisfying the conditions in {\bf (B)} 
  \item[Step 3.] Prove  \eqref{eq:B-cond2}
\end{itemize}

For Step 1, it is enough to show the identity in \eqref{eq:B-cond} for the restrictions of polynomial differential forms on any $T \in \mathcal{T}_h$. Moreover, by scaling argument, it suffices to show the equality on the reference element $\hat{T}$, i.e., 
\algns{
  \LRa{ u, v }_{h, \hat{T}} = \LRa{ u, v}, \qquad u \in \mc{Q}_r^- \Lm^k(\hat{T}), \quad v \in \mc{Q}_{r-1} \Lm^k (\hat{T}) .
}
In fact, this equality is true because the Gauss-Lobatto quadrature rule with $(r+1)$ points give exact integration for polynomials of degree $2r-1$. This completes the proof of Step~1.

The following result gives a proof of Step 2. 
\begin{theorem}
Let $\Pi_h:\tilde{\mc{Q}}_r \Lm^k(\mc{T}_h) \ra \mc{Q}_r^- \Lm^k(\mc{T}_h)$ be the interpolation operator defined by the canonical degrees of freedom of $\mc{Q}_r^- \Lm^k(\mc{T}_h)$, i.e., $\Pi_h u$ for  $u \in \tilde{\mc{Q}}_r \Lm^k(T)$ is characterized by 
\algn{ \label{eq:Qrminus-DOF}
\int_f \tr_f \Pi_h u \wedge v  &= \int_f \tr_f u \wedge v, & & v \in \mc{Q}_{r-1}^- \Lm^{l-k} (f), f \in \lap_l(\mathcal{T}_h), l \ge k.
}
Then $\Pi_h$ is bounded in $L^2\Lm^k(\Omega)$ with a norm independent of $h$, and  $\dd (u - \Pi_h u) = 0$ for $u \in \tilde{\mc{Q}}_r \Lm^k(\mathcal{T}_h)$.
\end{theorem}
\begin{proof}
Note the identity  
\algns{
\mc{Q}_r^- \Lm^k = \bigoplus_{\s \in \Sigma_k} \mc{Q}_{\s, r-1} \otimes \mc{Q}_{\s^*, r} \dd x_{\s} 
}
from the characterization of $\mc{Q}_r^- \Lm^k$. By the definitions of the degrees of freedom of
$\mc{Q}_r^- \Lm^k$ and $\tilde{\mc{Q}}_r \Lm^k$, the degrees of freedom of $\mc{Q}_r^- \Lm^k(\mathcal{T}_h)$ is a subset of the degrees of freedom of $\tilde{\mc{Q}}_r \Lm^k (\mathcal{T}_h)$. Therefore the uniform $L^2$ boundedness of $\Pi_h$ is a consequence of equivalence of the $L^2$ norm and a discerete norm defined by the degrees of freedom on each of these spaces.

To show $\dd ( u- \Pi_h u)  = 0$, without loss of generality, 
we consider $u$ defined only on one element $T$.  
Since $\dd ( u- \Pi_h u) \in \mc{Q}_r^- \Lm^{k+1}(T)$, it is sufficient to show that 
\algn{ \label{eq:du-tr-DOF}
\int_f \tr_f \dd (u - \Pi_h u) \wedge v = 0 \qquad \forall v \in \mc{Q}_{r-1}^- \Lm^{l - k - 1} (f), f \in \lap_{l}(T), l \ge k+1 
}
by the canonical degrees of freedom of $\mc{Q}_r^- \Lm^{k+1} (T)$. 
These vanishing identities follow from the commutativity of $\dd$ and $\tr_f$, and Stokes' theorem by
\algns{ 
\int_f \tr_f \dd (u - \Pi_h u) \wedge v &= \int_f \dd \tr_f (u - \Pi_h u) \wedge v \\
&= \int_{\pd f} \tr_{\pd f} \tr_f (u - \Pi_h u) \wedge \tr_{\pd f} v + \int_f \tr_f (u - \Pi_h u) \wedge \dd v \\
&= 0 
}
where the last equality follows from $\tr_{\pd f} \tr_f = \tr_{\pd f}$, the hierarchical trace property of $\mc{Q}_r^- \Lm^k$ spaces, the inclusion $\dd v \in \mc{Q}_{r-1}^- \Lm^{l-k} (f)$, and \eqref{eq:Qrminus-DOF}. 
\end{proof}

For Step 3, we can reduce \eqref{eq:B-cond2} to the corresponding identity on $\hat{T}$, i.e., it is enough to show
\algns{
\LRa{ u - \Pi_h u, v}_{h, \hat{T}} = 0
}
for $u \in \tilde{\mc{Q}}_r \Lm^k(\hat{T})$ and $v \in \mc{Q}_{r-1} \Lm^k (\hat{T})$. 

Before we start its proof, recall that the quadrature nodes of the Gauss-Lobatto rule with $(r+1)$ points
are the zeros of $\frac{d}{dt} L_r(t)$ on $[-1, 1]$. We also note that 
\algns{ 
(r+1)(L_{r+1}(t) - L_{r-1}(t) ) = (2r+1)(t L_r (t) - L_{r-1}(t)) = (2r+1) \frac{t^2 - 1}{r} \frac{d}{dt} L_r(t),
}
so 
\algn{ \label{eq:legendre-identity}
L_{r+1}(\xi^j) - L_{r-1}(\xi^j) = 0 \qquad 0 \le j \le r ,
}
i.e., the evaluation of $ L_{r+1}(t) - L_{r-1}(t)$ at the quadrature nodes of the Gauss-Lobatto rule with $(r+1)$ points vanishes. 

As $\Pi_h$, we define $\hat{\Pi}_h: \tilde{\mc{Q}}_r \Lm^k (\hat{T}) \ra \mc{Q}_r^- \Lm^k(\hat{T})$ as
\algns{
\int_f u \wedge v = \int_f \hat{\Pi}_h u \wedge v \qquad v \in \mc{Q}_{r-1}^- \Lm^{l-k} (f), f \in \lap_l(\hat{T}), l \ge k. 
}

\begin{theorem} \label{thm:intp-form}
Let $v = u - \hat{\Pi}_h u$ for $u \in \tilde{\mc{Q}}_r \Lm^k(\hat{T})$. In $v = \sum_{\s \in \Sigma_k(\hat{T})} v_\s \dd x_\s$ every $v_\s$
has a form 
\algn{ \label{eq:u-piu-form}
v_\s = \sum_{i \in \jump{\s^*}} \LRp{L_{r+1}(x_i) - L_{r-1}(x_i) } p_{\s, i} + v_{\s, 0}
}
where $p_{\s, i} \in \mc{Q}_{\s^*, r} (\hat{T}) \otimes \mc{Q}_{\s, r-1} (\hat{T})$, and every term in $v_{\s,0}$ written with the Legendre polynomials has a factor of 
$L_r(x_j)$ for some $j \in \jump{\s}$. 
\end{theorem}
\begin{proof}
Let $v = u - \hat{\Pi}_h u = \sum_{\s \in \Sigma_k(\hat{T})} v_\s \dd x_\s$ be given. 
To prove the assertion by induction we need to show the following two results.
First, for $f \in \lap_k(\hat{T})$ every polynomial coefficient of $\tr_f v \in \tilde{\mc{Q}}_r \Lm^k(f)$ satisfies \eqref{eq:u-piu-form}.
Second, if every polynomial coefficient of $\tr_f v$ satisfies \eqref{eq:u-piu-form} for all $f \in \lap_l(\hat{T})$ with given $l \ge k$, 
then $\tr_g v$ satisfies the same property corresponding to \eqref{eq:u-piu-form} for all $g \in \lap_{l+1}(\hat{T})$. 

We prove the first claim. Recalling $\tr_f v \in \tilde{\mc{Q}}_r \Lm^k (f) = \mc{Q}_r \Lm^k (f)$ for $f \in \lap_k(T)$ and the definition of $\hat{\Pi}_h$, the polynomial coefficient of $\tr_f v$ is a polynomial in $\mc{Q}_r (f)$ which is orthogonal to all $\mc{Q}_{r-1}(f)$
in the $L^2$ inner product. Then it has at least one factor of $L_r(x_j)$ for $j \in \jump{\s}$, therefore it is a form of \eqref{eq:u-piu-form}. 

For the second claim, from the trace property and the definition of $\hat{\Pi}_h$ commuting with the trace operator $\tr$, we can reduce the claim to the case $l = n-1$ without loss of generality. Suppose that every polynomial coefficient of $\tr_f v$ satisfies \eqref{eq:u-piu-form} for all $f \in \lap_{n-1}(\hat{T})$. 

By \eqref{eq:Qrminus-DOF} and the definition of shape function space of $\tilde{\mc{Q}}_r \Lm^k(\hat{T})$, $v_\s$ has a form
\algns{
\sum_{l \in \jump{\s^*}} {L_{r+1}(x_l) p_{\s, l} } + v_{\s, 0} + v_{\s, 1}
}
where $p_{\s,i} \in (\mc{Q}_{\s^*, r} \otimes \mc{Q}_{\s,r-1}) (\hat{T})$ but $p_{\s,i}$ is independent of $x_i$, $v_{\s,0}$ is as in the assertion, and $v_{\s,1}$ 
is a polynomial such that every term in its expression with the Legendre polynomials, has a factor of $L_r(x_l)$ or $L_{r-1}(x_l)$ for some $l \in \jump{\s^*}$. 
We remark that $v_{\s,0}$ may have polynomial terms of degree greater than $r$.
Let us rewrite $v_\s$ as 
\algns{
v_\s = \sum_{l \in \jump{\s^*}} \LRp{L_{r+1}(x_l) - L_{r-1}(x_l) } p_{\s, l}  + v_{\s, 0} + q_{\s} 
}
where $q_\s \in (\mc{Q}_{\s^*, r} \otimes \mc{Q}_{\s, r-1}) (\hat{T})$.
Note that $q_\s$ is a polynomial such that every term in its Legendre polynomial expression has $L_r(x_l)$ factor for some $l \in \jump{\s^*}$ or is in $\mc{Q}_{r-1}(\hat{T})$. 
Let $f$ be an $(n-1)$-dimensional hyperspace determined by $x_i = 1$ for some $i \in \jump{\s}$. Then 
the polynomial coefficient of the differential form $\dd x_\s$ in $\tr_f v$ is 
\algns{
\sum_{l \in \jump{\s^*}, l \not = i} \LRp{L_{r+1}(x_l) - L_{r-1}(x_l) } p_{\s, l}|_{x_i=1} + v_{\s,0}|_{x_i = 1} + q_{\s}|_{x_i = 1}. 
}
Since $\tr_f v$ can be written in the form of \eqref{eq:u-piu-form}, the comparison of the above and expressions in the form of \eqref{eq:u-piu-form} on $f$ leads to $q_{\s}|_{x_i = 1} = 0$. 
Since it holds for any $i \in \jump{\s^*}$ with $x_i = \pm 1$, 
\algns{
q_\s = \tilde{q}_{\s} \prod_{i \in \jump{\s^*}} (1-x_i^2), \qquad \tilde{q}_\s \in (\mc{Q}_{\s^*, r-2} \otimes \mc{Q}_{\s, r-1}) (\hat{T}). 
}
If we set $\eta = \tilde{q}_{\s} \dd x_{\s^*}$, then 
$\int_{\hat{T}} v \wedge \eta = 0$ due to the fact $v = u - \hat{\Pi}_h u$ and the definition of $\hat{\Pi}_h$. However, 
\algns{
\int_{\hat{T}} v \wedge \eta = \pm  \int_{\hat{T}} v_\s \tilde{q}_\s \vol_{\hat{T}} = \pm  \int_{\hat{T}} \prod_{i \in \jump{\s^*}} (1-x_i^2) \tilde{q}_\s^2 \vol_{\hat{T}}, 
}
therefore ${q}_\s = 0$. As a consequence $v_\s$ has a form \eqref{eq:u-piu-form}. 
\end{proof}

\begin{cor}
For $u \in \tilde{\mc{Q}}_r \Lm^k(\mc{T}_h)$ and $v \in \mc{Q}_{r-1}^d \Lm^k(\mc{T}_h)$, 
\algns{
\LRa{ u - \Pi_h u , v}_h = 0 .
}
\end{cor}
\begin{proof}
It suffices to show the assertion for  $u \in \tilde{\mc{Q}}_r \Lm^k(T)$ and $v \in \mc{Q}_{r-1}^d \Lm^k(T)$ for any $T \in \mc{T}_h$. We define $\hat{u} \in \tilde{\mc{Q}}_r \Lm^k(\hat{T})$ and $\hat{v} \in \mc{Q}_{r-1}^d \Lm^k(\hat{T})$ as the pullbacks $\phi^* u$ and $\phi^* v$
with the dilation map $\phi : \hat{T} \ra T$. From the definition of $\Pi_h$ and the property of pullback on differential forms, $\phi^* (u - \Pi_h u) = \hat{u} - \hat{\Pi}_h \hat{u}$ holds.
Since $\LRa{ u - \Pi_h u, v}_{h,T}$ is a constant multiple of $\LRa{ \hat{u} - \hat{\Pi}_h \hat{u}, \hat{v} }_{h, \hat{T}}$ with the constant depending on scaling.
By the characterization of $\hat{u} - \hat{\Pi}_h \hat{u}$ in Theorem~\ref{thm:intp-form}, the exactness of the Gauss-Lobatto quadrature rule for polynomials of degree $2r-1$, the orthogonality of Legendre polynomials, and \eqref{eq:legendre-identity}, $\LRa{ \hat{u} - \hat{\Pi}_h \hat{u}, \hat{v} }_{h, \hat{T}} = 0$. 
\end{proof}

Finally, we check that the proposed method give local coderivatives with an argument completely analogous to the one in \cite{Lee-Winther}.

To see more details, let the set of nodal degrees of freedom and the basis of $\tilde{\mc{Q}}_r \Lm^k(\mc{T}_h)$ associated to the nodal degrees of freedom be
\algns{
\{ \phi_{(T, z)} \,:\,  z \in N_r(T), T \in \mc{T}_h \} , \qquad 
\{ \psi_{(T, z)} \,:\,  z \in N_r(T), T \in \mc{T}_h \} ,
}
such that $\phi_{(T,z)} (\psi_{(T', z')}) = \delta_{T T'} \delta_{z z'}$ with the Kronecker delta. 
For each $z \in N_r(T)$ for some $T \in \mc{T}_h$ and we can consider the set of all basis functions $ \psi_{(T', z)}$ with the common point $z$. We denote this set of basis functions by $\Psi_z$. Then the quadrature bilinear form $\LRa{\cdot, \cdot}_h$ with $\{ \psi_{(T, z)} \}$
gives a block diagonal matrix such that each block is associated to $\Psi_z$ for some $z$. 
The influence of the inverse of this block matrix associated to $\Psi_z$ is only on the local domain consists of $\{T'\}$ such that $\psi_{(T', z)} \in \Psi_z$, so it is a spatially local operator and then gives a local coderivative.

\section{Concluding remarks} \label{conclusion}
In this paper we develop high order numerical methods with finite elements for the Hodge Laplace problems on cubical meshes that admit local approximations of the coderivatives. The main task of the development is construction of a new family of finite element differential forms on cubical meshes which satisfy the properties for the stability analysis framework presented in Section~2. 


In our study we have only considered Hodge-Laplace problems with the identity coefficients. However, 
the discussion can be easily extended to a matrix-valued coefficient $\bs{K}^{-1}$ which is symmetric positive definite, and piecewise constant on $\mc{T}_h$. We point out that this is an important difference between our approach and the work in \cite{MonkCohen1998}
for the Maxwell equations. More precisely, the quadrature rules in \cite{MonkCohen1998} are combinations of Gauss and Gauss-Lobatto rules for different components, so the methods are not robust for matrix-valued coefficients because the matrix-valued coefficients in bilinear forms mix the components with different quadrature rules. We refer the interested readers to \cite[Section~6]{Lee-Winther} for details of the analysis with $\bs{K}^{-1} \not = I$.

Finally, the extension of the methods to curvilinear meshes is not easy because the conditions in {\bf{(B)}} strongly rely on the properties of quadrature rules on cubical meshes which do not hold on distorted cubical (or called curvilinear) meshes.
In \cite{Lee-Yotov-etal}, this extension was done for weakly distorted quadrilateral and hexahedral meshes by measuring changes of bilinear forms and related quantities under mesh distortion. However, this is possible in the cases because explicit forms of the bilinear and trilinear maps and their derivatives are available for the specific setting $k = n-1$ and $n =2,3$. The extension to general $n$ and $k$ is completely open.

\providecommand{\bysame}{\leavevmode\hbox to3em{\hrulefill}\thinspace}
\providecommand{\MR}{\relax\ifhmode\unskip\space\fi MR }
\providecommand{\MRhref}[2]{%
  \href{http://www.ams.org/mathscinet-getitem?mr=#1}{#2}
}
\providecommand{\href}[2]{#2}

\bibliographystyle{amsplain}
\vspace{.125in}

\end{document}

%% file: jlmacros.tex
\newcommand{\pmat}[1]{\begin{pmatrix} #1 \end{pmatrix}}

\newcommand{\beq} {\begin{equation}}
\newcommand{\eeq} {\end{equation}}
\newcommand{\bdm} {\begin{displaymath}}
\newcommand{\edm} {\end{displaymath}}

\newcommand{\bit}{\begin{itemize}}
\newcommand{\eit}{\end{itemize}}
\newcommand{\bde}{\begin{description}}
\newcommand{\ede}{\end{description}}

\newcommand{\ben}{\begin{enumerate}}
\newcommand{\een}{\end{enumerate}}
\newcommand{\algn}[1]{\begin{align} #1 \end{align}}
\newcommand{\algns}[1]{\begin{align*} #1 \end{align*}}

\newcommand{\gats}[1]{\begin{gather*} #1 \end{gather*}}
\newcommand{\barr}{\begin{array}}
\newcommand{\earr}{\end{array}}

\newcommand{\mc}[1]{\mathcal{#1}}
\newcommand{\LRp}[1]{\left( #1 \right)}

\newcommand{\LRa}[1]{\left< #1 \right>}

\newcommand{\jump}[1] {\ensuremath{\left[\![{#1}]\!\right]}}

\newcommand{\ra}{\rightarrow}

\newcommand{\e}{\epsilon}

\newcommand{\R}{{\mathbb R}}

\newcommand{\grad}{\operatorname{grad}}

\renewcommand{\div}{\operatorname{div}}

\newcommand{\tr}{\operatorname{tr}}

\newcommand{\spn}{\operatorname{span}}

\newcommand{\bs}{\boldsymbol}

\newcommand{\lap}{\Delta}
\newcommand{\pd}{\partial}

\newtheorem{theorem}{Theorem}[section]
\newtheorem{prop}[theorem]{Proposition} 
\newtheorem{lemma}[theorem]{Lemma}
\newtheorem{cor}[theorem]{Corollary}
